\documentclass{amsart}

\usepackage{adjustbox}
\usepackage{amsaddr}
\usepackage{amsmath}
\usepackage{amssymb}
\usepackage{bm}
\usepackage{booktabs}
\usepackage{enumitem}
\usepackage[margin=1in]{geometry}
\usepackage{hyperref}
\usepackage{cleveref} % must be after hyperref
\usepackage{mathtools} % multlined
\usepackage{siunitx}
\usepackage{stmaryrd}
\usepackage{tikz}
\usepackage{tikz-cd}

\usetikzlibrary{calc}

\usepackage[giveninits,date=year,sortcites]{biblatex}
\newcolumntype{C}[1]{S[table-format=#1,table-number-alignment=center]}

\newcolumntype{?}{!{\vrule width 1pt}}

\theoremstyle{plain}
\newtheorem{proposition}{Proposition}
\newtheorem{lemma}{Lemma}
\newtheorem{corollary}{Corollary}
\newtheorem{theorem}{Theorem}
\newtheorem{definition}{Definition}
\newtheorem*{lemma*}{Lemma}

\theoremstyle{remark}
\newtheorem{remark}{Remark}

\newcommand{\Hdiv}{\bm{H}(\operatorname{div})}
\newcommand{\HdivOmega}{\bm{H}(\operatorname{div},\Omega)}

\DeclareMathOperator{\Curl}{curl}
\DeclareMathOperator{\Div}{div}

\newcommand{\Hcurl}{\bm{H}(\Curl)}
\newcommand{\HcurlOmega}{\bm{H}(\Curl,\Omega)}
\newcommand{\HcurlDomain}[1]{\bm{H}(\Curl,#1)}

\newcommand{\T}{\mathcal{T}}
\newcommand{\kk}{\kappa}

\DeclareMathOperator{\nnz}{nnz}

\newcommand{\Vh}{V_h}
\newcommand{\Wh}{\bm{W}_h}
\newcommand{\Zh}{Z_h}

\newcommand{\PP}{\mathbb{P}}
\newcommand{\Qq}{\mathcal{Q}}
\newcommand{\Pp}{\mathcal{P}}

\newcommand{\tr}{T}

\def\iii{\vert\kern-0.25ex\vert\kern-0.25ex\vert}

\DeclareMathAlphabet\mathbfcal{OMS}{cmsy}{b}{n}

\SetSymbolFont{stmry}{bold}{U}{stmry}{m}{n}

\newcommand{\Tri}{\Delta}
\newcommand{\Square}{\Xi}

\newcommand{\ii}{\overline{\imath}}
\newcommand{\jj}{\overline{\jmath}}

\begin{document}

\title[Finite element Duffy de Rham Complex]{The high-order finite element Duffy de Rham complex and low-order-refined preconditioning}
\author[W. Pazner]{Will Pazner}
\address{Fariborz Maseeh Department of Mathematics and Statistics, Portland State University, Portland, OR}

\begin{abstract}
   In this work, we construct high-order finite element spaces for the $L^2$ de Rham complex on triangular meshes amenable to low-order-refined preconditioning.
   The spaces are constructed using the Duffy transformation, by pulling back appropriately chosen polynomial spaces defined on the unit square;
   in addition to piecewise polynomials, these spaces also contain certain rational functions, and they reduce to the standard Lagrange, Nédélec, and discontinuous finite elements in the lowest-order case.
   We establish spectral equivalence, independent of the polynomial degree, of the stiffness matrices defined on these spaces with the lowest-order stiffness matrices defined on refined meshes, constructed using a Gauss--Lobatto triangular lattice.
   Spectral equivalence of the operators is a consequence of norm equivalences in Jacobi-weighted $L^2$ norms, which are established by proving stability of the Jacobi--Gauss--Lobatto interpolation operator in shifted norms.
   The low-order-refined preconditioners can also be used to precondition the standard piecewise polynomial finite element spaces using a fictitious space approach.
   The low-order-refined system can in turn be preconditioned effectively using algebraic multigrid methods.
   The analytical estimates are confirmed by numerical results on a variety of high-order problems, including on mixed meshes and surface meshes.
\end{abstract}

\maketitle

\section{Introduction}
\label{sec:intro}

This paper is concerned with the construction of high-order finite elements for the $L^2$ de Rham complex on triangular meshes that are amenable to low-order-refined preconditioning.
Low-order-refined preconditioning, also known as finite-element-method--spectral-element-method (FEM--SEM) preconditioning, is a technique where high-order discrete operators are shown to be spectrally equivalent to certain lowest-order operators, independent of the polynomial degree of the high-order operator, and hence preconditioners or solvers for the low-order system can be used as effective preconditioners for the high-order system;
this idea was first proposed by Orszag in 1980 in the context of spectral methods \cite{Orszag1980}.
Since then, there have been many further developments and improvements \cite{Canuto2010,Bello-Maldonado2019}, including application to the incompressible Navier--Stokes equations \cite{Fischer1997}, extension to discontinuous Galerkin methods \cite{Pazner2020a}, and application to interior penalty and saddle-point problems in $\Hdiv$ \cite{Pazner2025,Pazner2024}.
A systematic methodology for constructing spectrally equivalent discretizations for all spaces in the $L^2$ de Rham complex based on an interpolation--histopolation methodology was described in \cite{Pazner2023}.
The main advantage of low-order-refined preconditioning is that preconditioners for the high-order system can be constructed without ever assembling the associated system matrix.
Then, efficient matrix-free algorithms can be used to compute the action of the operator, resulting in significant asymptotic improvements to the computational complexity and memory requirements of the iterative solvers.

For the most part, the above-mentioned works are applicable only to discretizations defined on tensor-product elements, i.e.\ meshes of mapped quadrilaterals and hexahedra.
One notable exception is the work of Chalmers and Warburton \cite{Chalmers2018}, who considered an optimization procedure to choose the low-order node locations so as to minimize the condition number of the preconditioned system.
In contrast to the tensor-product case, in order to obtain a condition number that is bounded independent of the polynomial degree, the low-order discretization needed to be enriched with additional nodal points, leading to a non-identity transfer operator between the high-order and low-order spaces.

In this work, we take a somewhat different approach.
Instead of constructing low-order discretizations for the standard $\Pp_N$ finite elements (and the corresponding Nédélec and $L^2$ elements), we instead construct different elements, using the so-called Duffy transformation to map certain tensor-product elements to elements on the triangle.
Each resulting local element contains the complete polynomial space, giving standard best-approximation properties, but also certain rational functions.
A very similar construction was considered for $H^1$ elements in \cite{Chen2012}.
The approximation properties of rational functions obtained from polynomials under the Duffy transformation were analyzed in \cite{Shen2009}.
The Duffy transformation and Dubiner polynomials were also used in \cite{Sherwin1995,Karniadakis2005} to construct modal bases for $hp$ spectral element discretizations on simplex elements.

We also construct, using a unified methodology, finite element spaces for the $\Hcurl$ and $L^2$ spaces.
With a judicious choice of the degrees of freedom for these spaces, the discrete differential operators between these spaces are exactly the topological incidence matrices.
In the lowest-order case, these new elements coincide exactly with the standard lowest-order $\Pp_N$, Nédélec first-kind, and $L^2$ elements on the triangle.

The main result of this paper is that the system matrices for the high-order Laplacian, curl--curl operator, and $L^2$ mass matrix using the newly defined elements are spectrally equivalent, independent of the polynomial degree, to their lowest-order counterparts defined on a refined mesh obtained by superimposing a collapsed triangular lattice on each element.
The key technical contribution is the $H^1$ stability of Jacobi--Gauss--Lobatto interpolation operators in certain shifted norms, i.e.\ where the value of the Jacobi exponents of the interpolation points may differ from those of the norm considered.
Change of variables under the Duffy transformation gives rise to a Jacobi weighting, necessitating norm equivalence in weighted norms, proved using these stability results.

The structure of this paper is as follows.
In \Cref{sec:fem}, we define the finite element Duffy de Rham complex on triangular meshes, and construct unisolvent degrees of freedom for each of the elements.
In \Cref{sec:lor}, we construct low-order-refined preconditioners for the high-order elements, and prove spectral equivalence using a detailed analysis of polynomial interpolation in Jacobi-weighted Sobolev spaces.
Numerical results confirming the analysis are provided in \Cref{sec:numerical}.
Conclusions are given in \Cref{sec:conclusions}.

\begin{remark}[Notation]
   We will use the notation $a \lesssim b$ to express that $a \leq C b$, where $C$ is a constant independent of the mesh size $h$ and polynomial degree $N$.
   Similarly, $a \gtrsim b$ is used to mean $b \lesssim a$, and $a \approx b$ means that both $a \lesssim b$ and $b \lesssim a$.
   The standard norms and seminorms on the Sobolev space $H^k(D)$ will be denoted $\|\cdot\|_{k,D}$ and $|\cdot|_{k,D}$, respectively;
   when the subscript $D$ is omitted, the domain of integration will be assumed to be all of $\Omega$.

   Let $\PP^N$ denote the space of polynomials in one variable of degree at most $N$.
   Let $\Qq_{M,N}$ denote the space of bivariate polynomials of degree at most $M$ in the first variable, and at most $N$ in the second variable, and define $\Qq_N = \Qq_{N,N}$.
   Let $\Pp_N$ denote the space of bivariate polynomials of total degree at most $N$.
   For multi-index $\bm\alpha \in \mathbb{N}^d$ and $\bm x \in \mathbb{R}^d$, define $|\bm\alpha| = \sum_{i=1}^d \alpha_i$ and $\bm x^{\bm \alpha} := \prod_{i=1}^d x_i^{\alpha_i}$.
\end{remark}

\section{Finite element spaces}
\label{sec:fem}

We consider the two-dimensional $L^2$ de Rham complex
\begin{equation}
   \label{eq:de-rham}
   \begin{tikzcd}
      H^1(\Omega) \arrow[r, "\nabla"] & \HcurlOmega \arrow[r, "\Curl"] & L^2(\Omega),
   \end{tikzcd}
\end{equation}
and the rotated version
\begin{equation}
   \label{eq:de-rham-rotated}
   \begin{tikzcd}
      H^1(\Omega) \arrow[r, "\nabla^\perp"] & \HdivOmega \arrow[r, "\Div"] & L^2(\Omega).
   \end{tikzcd}
\end{equation}
Here, $\Curl \bm w = \partial_x w_2 - \partial_y w_1$ and $\nabla^\perp v = (-\partial_y v, \partial_x v)$.
For our purposes, \eqref{eq:de-rham} and \eqref{eq:de-rham-rotated} are largely interchangeable, and we will focus on \eqref{eq:de-rham}.

Given two domains $E$ and $F$, and a diffeomorphism $\theta : E \to F$, a function defined on $F$ is transformed to a function defined on $E$ through the pullback under $\theta$.
The pullbacks associated with the $H^1$, $\Hcurl$, and $L^2$ spaces are defined by
\begin{align}
   \theta^*_g v &= v \circ \theta \in H^1(E), \\
   \theta^*_c \bm w &= (D\theta)^\tr \bm w \circ \theta \in \HcurlDomain{E}, \\
   \theta^*_b z &= \det(D\theta) z \circ \theta \in L^2(E),
\end{align}
where $v \in H^1(F)$, $\bm w \in \HcurlDomain{F}$, and $z \in L^2(F)$.
In the above, the subscripts $g$, $c$, and $b$ stand for `gradient', `curl', and `broken', following the notation of \cite[Definition 9.8]{Ern2021}.
The pullbacks commute with the differential operator as summarized in the following proposition.

\begin{proposition}[Commutativity of differential and pullback]
   It holds that
   \begin{align}
      \label{eq:grad-commute}
      \nabla(\theta^*_g u) &= \theta^*_c (\nabla u), \\
      \label{eq:curl-commute}
      \Curl(\theta^*_c \bm v) &= \theta^*_b (\Curl \bm v).
   \end{align}
\end{proposition}

\subsection{Duffy transformation}

The Duffy transformation was introduced by M.~Duffy in \cite{Duffy1982} to transform integrals over the square pyramid to integrals over the cube.
In this work, we use a similar transformation to map the unit square $\Square = [0,1]^2$ to the unit right triangle $\Tri$ with vertices $(0,0)$, $(1,0)$, and $(0,1)$.
Define $\varphi : \Square \to \Tri$ and its inverse (away from the top vertex) $\psi = \varphi^{-1} : \Tri \to \Square$ by
\begin{align}
   \label{eq:duffy-phi}
   \varphi(x,y) &= (x(1-y), y),\\
   \label{eq:duffy-psi}
   \psi(x,y) &= (x/(1-y), y).
\end{align}
The Jacobians of these mappings are given by
\begin{alignat}{3}
   D\varphi &= \begin{pmatrix}
      1-y & -x \\
      0 & 1
   \end{pmatrix}, &\qquad& \det(D\varphi) = 1-y, \\
   D\psi &= \begin{pmatrix}
      1/(1-y) & x/(1-y)^2 \\
      0 & 1
   \end{pmatrix}, &\qquad& \det(D\psi) = \frac{1}{1-y}.
\end{alignat}

\subsection{Construction of spaces}
\label{sec:spaces}

Consider a triangular mesh $\T = \{ \kk \}$.
Every triangle $\kk \in \T$ is the image of the reference triangle $\Tri$ under an affine transformation $T_\kk : \Tri \mapsto \kk$.
We construct conforming finite elements for the complex \eqref{eq:de-rham} consisting of both piecewise polynomials and certain rational functions;
the spaces are parameterized by a polynomial degree $N \in \mathbb{N}$.
The spaces are defined in terms of local function spaces on the triangles $\kk \in \T$,
\begin{align}
   \label{eq:Vh}
   \Vh &= \{ v \in H^1(\Omega) : v|\kk \in V_N(\kk) \}, \\
   \label{eq:Wh}
   \Wh &= \{ \bm w \in \HcurlOmega : \bm w|\kk \in \bm W_N(\kk) \}, \\
   \label{eq:Zh}
   \Zh &= \{ z \in L^2(\Omega) : z|\kk \in Z_N(\kk) \}.
\end{align}
Let $X(\kk)$ denote one of $V_N(\kk), \bm W_N(\kk), Z_N(\kk)$, and correspondingly, let $H^\bullet$ denote the associated space among $H^1, \Hcurl, L^2$.
The local space $X(\kk)$ is defined as the pullback under the element transformation $T_\kk$ of a reference space $\widehat{X} = X(\Tri)$.
The local reference space is constructed by pulling back polynomials under the Duffy transformation $\psi$ defined in \eqref{eq:duffy-psi}, i.e.\ $\widehat{X} = \psi^*_\bullet \widetilde{X}$, where $\psi^*_\bullet$ is one of $\psi^*_g, \psi^*_c, \psi^*_b$ according to $H^\bullet$.
The \textit{reference polynomial space} $\widetilde{X}$ is a space of bivariate polynomials chosen such that the conformity condition $\widehat{X} \subseteq H^\bullet(\Tri)$ holds.
In what follows, we write these conditions explicitly for each of the spaces considered.

\subsection{\texorpdfstring{$H^1$}{H1} local space}

We define the local space $\widehat{V}_N = \psi^*_g \widetilde{V}_N$, where $\widetilde{V}_N \subseteq \Qq_N$.
The condition that $V_N(\Tri) \subseteq H^1(\Tri)$ determines the space of permissible polynomials $\widetilde{V}_N$.
For any $v \in \Qq_N$, $\psi^*_g v \in L^2(\Tri)$, and so the permissible polynomials are those that satisfy $\nabla(\psi^*_g v) \in L^2(\Tri)$.
By \eqref{eq:grad-commute}, it holds that
\[
   \nabla (\psi^*_g v) = \psi^*_c (\nabla v) = (D\psi)^\tr (\nabla v) \circ \psi.
\]
The $L^2$ norm of the gradient can be computed conveniently by changing coordinates to integrate over the unit square,
\begin{equation}
   \label{eq:h1-computation}
   \begin{aligned}
      \| \nabla (\psi^*_g v) \|_{\bm L^2(\Tri)}^2
         &= \int_{\Tri} \| (D\psi)^\tr (\nabla v) \circ \psi \|^2 \, d\bm x
         = \int_{\Square} \| ((D\psi)^\tr \circ \varphi) (\nabla v) \|^2 \det(D\varphi) \, d\bm x \\
         &= \int_{\Square} \left\|
            \begin{pmatrix}
               1/(1-y) & 0 \\
               x/(1-y) & 1
            \end{pmatrix}
            \begin{pmatrix} \partial_x v \\ \partial_y v \end{pmatrix}
         \right\|^2 (1-y) \, d\bm x.
   \end{aligned}
\end{equation}
From this, it is clear that $\nabla (\psi^*_g v) \in \bm L^2(\Tri)$ if and only if $1-y$ divides $\partial_x v$.
In other words, $v(x,y) = (1-y) q(x,y) + r$; by Euclidean division, the remainder is either zero or degree-zero, so $r$ is a constant.
\begin{definition}
   \label{def:h1}
   The $H^1$ reference polynomial space is given by
   \begin{equation}
      \label{eq:h1-defn}
      \widetilde{V}_N = (1-y) \Qq_{N,N-1} + \Qq_0.
   \end{equation}
\end{definition}
\begin{proposition}
   \label{prop:Vh-polyn}
   The degree-$N$ local $H^1$ space $\widehat{V}_N$ contains $\Pp_N$, all polynomials of total degree at most $N$.
\end{proposition}
\begin{proof}
   Let $p \in \Pp_N$, and write $p(\bm x) = \sum_{\bm \alpha} c_{\bm \alpha} \bm x^{\bm \alpha}$.
   Since $p = \psi^*_g \varphi^*_g p$, it suffices to verify that $\varphi^*_g p \in \widetilde{V}_N$.
   To this end,
   \begin{align*}
      (\varphi^*_g p)(\bm x)
         &= p(\varphi(\bm x)) = p(x(1-y), y)
         = \sum_{\bm\alpha} c_{\bm\alpha} x^{\alpha_1} (1-y)^{\alpha_1} y^{\alpha_2} \\
         &= (1-y) \Big( \sum_{\bm\alpha, \alpha_1 \geq 1} c_{\bm\alpha} x^{\alpha_1} (1-y)^{\alpha_1 - 1} y^{\alpha_2} \Big) + \sum_{\bm\alpha, \alpha_1 = 0} c_{\bm\alpha} y^{\alpha_2}.
   \end{align*}
   Since $\alpha_1, \alpha_1 + \alpha_2 \leq N$, the first term on the right-hand side is in $(1-y)\Qq_{N,N-1}$.
   Since $\alpha_2 \leq N$, by Euclidean division by $1-y$, the second term on the right-hand side is in $(1-y)\Qq_{0,N-1} + \Qq_0$, and so $\varphi^*_g p \in \widetilde{V}_N$.
\end{proof}

\begin{proposition}
   \label{prop:h1-trace}
   An element of $\widehat{V}_N$ restricted to an edge of $\Tri$ is a polynomial of degree at most $N$.
\end{proposition}
\begin{proof}
   Let $v \in \widetilde{V}_N$.
   Then, $(\psi^*_g v)(t,0) = v(\psi(t,0)) = v(t,0) \in \PP^N$.
   Similarly, $(\psi^*_g v)(0,t) = v(0,t) \in \PP^N$.
   Finally, $(\psi^*_g v)(1-t,t) = v(1,t) \in \PP^N$.
\end{proof}

\subsection{\texorpdfstring{$\Hcurl$}{H(curl)} local space}
\label{sec:hcurl-space}

We define the polynomial space $\widetilde{\bm W} \subseteq \Qq_{N-1,N+1} \times \Qq_{N,N-1}$ such that $\widehat{\bm W} = \psi^*_c \widetilde{\bm W} \subseteq \HcurlDomain{\Tri}$.
Let $\bm w = (w_1, w_2)$, with $w_1(\bm x) = \sum_{\bm \alpha} c_{\bm\alpha} \bm x^{\bm \alpha}$, $w_2(\bm x) = \sum_{\bm \alpha} d_{\bm\alpha} \bm x^{\bm \alpha}$.
In contrast to the previous case, both conditions $\psi^*_c \bm w \in \bm L^2(\Tri)$ and $\Curl(\psi^*c \bm w) \in L^2(\Tri)$ are nonvacuous.
Indeed, a computation similar to \eqref{eq:h1-computation} gives that
\begin{align}
   \label{eq:hcurl-integral}
   \| \psi^*_c \bm w \|_{\bm L^2(\Tri)}^2
      &= \int_{\Square} \left\|
         \begin{pmatrix}
            1/(1-y) & 0 \\
            x/(1-y) & 1
         \end{pmatrix}
         \begin{pmatrix}w_1 \\ w_2\end{pmatrix}
      \right\|^2 (1-y) \, d\bm x,
\end{align}
and so $\psi^*_c \bm w \in L^2(\Tri)$ if and only if $1-y$ divides $w_1$, i.e.
\begin{equation}
   \label{eq:hcurl-v1-l2}
   w_1(x,y) = (1-y) q_1(x,y).
\end{equation}
Turning to the condition $\Curl(\psi^*_c \bm w) \in L^2(\Tri)$, by \eqref{eq:curl-commute} it holds that
\[
   \Curl(\psi^*_c \bm w) = \psi^*_b (\Curl \bm w).
\]
Then, changing coordinates to integrate over the unit square,
\begin{equation*}
   \label{eq:hcurl-computation}
   \begin{aligned}
      \| \Curl(\psi^*_c \bm w) \|_{\bm L^2(\Tri)}^2
         &= \int_{\Tri} (\det(D\psi) \Curl(\bm w) \circ \psi )^2 \, d\bm x
         = \int_{\Square} \det(D\psi)^2 (\Curl \bm w)^2 \det(D\varphi) \, d\bm x \\
         &= \int_{\Square} (\Curl \bm w)^2 / (1-y) \, d\bm x.
   \end{aligned}
\end{equation*}
This gives the condition that $1 - y$ must divide $\Curl \bm w = \partial_x w_2 - \partial_y w_1$, i.e.
\begin{align*}
   \partial_x w_2(x,y) - \partial_y w_1(x,y)
      &= (1-y) q_2(x, y)
\end{align*}
for some polynomial $q_2$.
Rearranging, this gives
\begin{equation}
   \label{eq:dy-v1-1}
   -\partial_y w_1(x,y) = (1-y) q_2(x,y) - \partial_x w_2(x,y).
\end{equation}
From \eqref{eq:hcurl-v1-l2}, we have that
\begin{equation}
   \label{eq:dy-v1-2}
   -\partial_y w_1(x,y) = -\partial_y ((1-y) q_1(x,y)) = q_1(x,y) - (1-y) \partial_y q_1(x,y).
\end{equation}
Combining \eqref{eq:dy-v1-1} and \eqref{eq:dy-v1-2}, we obtain
\begin{align*}
   q_1(x,y)
      &= (1-y) q_2(x,y) - \partial_x w_2(x,y) - (1-y) \partial_y q_1(x,y) \\
      &= (1-y) (q_2(x,y) - \partial_y q_1(x,y)) - \partial_x w_2(x,y).
\end{align*}
Define $s := q_2 - \partial_y q_1$.
Then,
\[
   q_1(x,y) = (1-y) s(x,y) - \partial_x w_2(x,y).
\]
Substituting into \eqref{eq:hcurl-v1-l2}, we obtain
\begin{equation}
   \label{eq:hcurl-v1-condition}
   w_1(x,y) = (1-y)^2 s(x,y) - (1-y) \partial_x w_2(x,y).
\end{equation}

\begin{definition}
   The $\Hcurl$ reference polynomial space is defined by
   \begin{equation}
      \label{eq:hcurl-defn}
      \widetilde{\bm W}_N = ((1-y)^2 \Qq_{N-1}) \times \{ 0 \}
         + \{ (-(1-y) \partial_x w_2, w_2) : w_2 \in \Qq_{N,N-1} \}.
   \end{equation}
\end{definition}

Just as the $H^1$ space $V_N$ contains the standard $\Pp_N$ Lagrange space, the $\Hcurl$ space $\bm W_N$ contains the first-kind Nédélec space $\bm N_{N-1}$, which is defined by
\[
   \bm N_{N-1} = \mathbfcal{P}_{N-1} \oplus \mathbfcal{S}_{N}, \qquad
   \mathbfcal{S}_{N} = \left\{ \bm s \in \mathbfcal{P}_{N}^H : \bm s \cdot \bm x = 0 \right\},
\]
where $\mathbfcal{P}_{N-1} = [\Pp_{N-1}]^2$, and $\mathbfcal{P}_{N}^H = [\Pp_N^H]^2$ is the space of vector-valued homogeneous polynomials,
\[
   \mathcal{P}_{N}^H = \Big\{ p : p(x,y) = \sum_{i+j = N} c_{ij} x^i y^j \Big\}.
\]
\begin{proposition}
   \label{prop:Wh-polyn}
   The degree-$N$ local $\Hcurl$ space $\bm W_N$ contains the Nédélec first-kind space $\bm N_{N-1}$, and consequently it contains all polynomials of total degree at most $N-1$.
\end{proposition}
\begin{proof}
   We first show $\mathbfcal{P}_{N-1} \subseteq \bm W_N$.
   Let $\bm p \in \mathbfcal{P}_{N-1}$.
   It suffices to show that $\bm w = \varphi^*_c \bm p \in \widetilde{\bm W}_N$.
   Writing $\bm p(\bm x) = (\sum_{\bm \alpha} b_{\bm\alpha} \bm x^{\bm \alpha}, \sum_{\bm \alpha} c_{\bm \alpha} \bm x^{\bm \alpha})$ where $|\alpha| \leq N - 1$,
   \begin{align*}
      \bm w(x,y) = (D\varphi)^\tr \bm p(\varphi(x,y))
      = \begin{pmatrix}
         \sum_{\bm\alpha} b_{\bm\alpha} x^{\alpha_1} (1-y)^{\alpha_1+1} y^{\alpha_2} \\
         \sum_{\bm\alpha} \left(-b_{\bm\alpha} x^{\alpha_1+1} (1-y)^{\alpha_1} y^{\alpha_2} + c_{\bm\alpha}x^{\alpha_1}(1-y)^{\alpha_1}y^{\alpha_2} \right)
      \end{pmatrix}.
   \end{align*}
   We see $w_2 \in \Qq_{N,N-1}$ since $\alpha_1, \alpha_1 + \alpha_2 \leq N - 1$.
   An explicit computation shows that $(1-y)^2$ divides $w_1 + (1-y)\partial_x w_2$, so $\bm w \in \widetilde{\bm W}_N$.

   Now, we show that for $\bm s \in \mathcal{S}_N$, $\bm v = \varphi^*_c \bm s \in \widetilde{\bm W}_N$.
   Writing $\bm s(x,y) = (\sum_{i + j = N} b_{ij} x^i y^j, \sum_{i + j = N} c_{ij} x^i y^j)$, we first note that the condition $\bm s \cdot \bm x = 0$ implies that
   \begin{equation}
      \label{eq:nedelec-coeff}
      b_{i-1,j} = -c_{i,j-1}, \qquad
      b_{ij} = -c_{i+1,j-1}, \qquad
      c_{ij} = -b_{i-1,j+1},
   \end{equation}
   where any coefficient out of bounds is taken to be zero.
   Computing the second component of $\bm v$,
   \begin{align*}
      v_2 &= \sum_{i + j = N} \left( -b_{ij} x^{i + 1} (1-y)^{i} y^{j} + c_{ij} x^{i} (1-y)^{i} y^{j} \right)
      = \sum_{i + j = N} \left( c_{i+1,j-1} x^{i + 1} (1-y)^{i} y^{j} + c_{ij} x^{i} (1-y)^{i} y^{j} \right) \\
      &= \sum_{i + j = N} \left( c_{i,j} x^{i} (1-y)^{i-1} y^{j+1} + c_{ij} x^{i} (1-y)^{i} y^{j} \right)
      = \sum_{i + j = N} c_{ij} x^i (1-y)^{i-1} y^j,
   \end{align*}
   and so $v_2 \in \Qq_{N,N-1}$.
   It remains to show that $v_1 + (1-y) \partial_x v_2 \in (1-y)^2 \Qq_{N-1}$.
   Using \eqref{eq:nedelec-coeff},
   \begin{align*}
      v_1 + (1-y) \partial_x v_2
         &= \sum_{i + j = N} \left( b_{ij} x^{i} (1-y)^{i + 1} y^{j} + i c_{ij} x^{i-1} (1-y)^{i} y^j \right) \\
         &= \sum_{i + j = N} \left( -c_{i+1,j-1} x^{i} (1-y)^{i + 1} y^{j} + i c_{ij} x^{i-1} (1-y)^{i} y^j \right) \\
         &= \sum_{i + j = N} \left( -c_{i,j} x^{i-1} (1-y)^{i} y^{j+1} + i c_{ij} x^{i-1} (1-y)^{i} y^j \right)
         = \sum_{i + j = N} c_{ij} x^{i-1} (1-y)^{i} y^j (1 - iy)
   \end{align*}
   We see that $(1-y)^2$ divides the above expression, since for $i = 0$, $c_{ij} = 0$, for $i=1$, $(1-y)^i(1-iy) = (1-y)^2$, and for $i > 1$, $(1-y)^2$ divides the factor $(1-y)^i$.
   Therefore, $v_1 + (1-y) \partial_x v_2 \in (1-y)^2 \Qq_{N-1}$, and $\bm v \in \widetilde{\bm W}_N$ as desired.
\end{proof}

\begin{proposition}
   \label{prop:hcurl-trace}
   The tangential components of $\bm w \in \widehat{\bm W}_N$ are polynomials of degree at most $N - 1$.
\end{proposition}
\begin{proof}
   Let $\bm w = \psi^*_c \tilde{\bm w}$, with $\tilde{\bm w} = (\tilde{w}_1, \tilde{w}_2) \in \widetilde{\bm W}_N$.
   On the bottom edge, the tangential trace is $(\tilde{w}_1(x,0), x \tilde{w}_1(x,0) + \tilde{w}_2(x,0)) \cdot \bm t = \tilde{w}_1(x,0) \in \PP^{N-1}$ by \eqref{eq:hcurl-defn}.
   Similarly, on the left edge, the tangential trace is $((1-y)^{-1}\tilde{w}_1(0,y), \tilde{w}_2(0,y)) \cdot \bm t = \tilde{w}_2(0,y) \in \PP^{N-1}$.
   Finally, on the diagonal edge, the trace is $((1-y)^{-1} \tilde{w}_1(1,y), (1-y)^{-1} \tilde{w}_1(1,y) + \tilde{w}_2(1,y)) \cdot \bm t = (1/\sqrt{2}) \tilde{w}_2(1,y) \in \PP^{N-1}$.
\end{proof}

\subsection{\texorpdfstring{$L^2$}{L2} local space}

Finally, we characterize polynomials $z(\bm x) = \sum_{\bm \alpha} c_{\bm \alpha} \bm x^{\bm \alpha}$ such that $\psi^*_b z \in L^2(\Tri)$.
As before, we compute
\begin{align*}
   \| \psi^*_b z \|_{L^2(\Tri)}^2
      = \int_{\Tri} (\det(D\psi) z \circ \psi)^2 \, d\bm x
      = \int_{\Square} (\det(D\psi) z )^2 \det(D\varphi) \, d\bm x
      = \int_{\Square} z^2/(1-y) \, d\bm x.
\end{align*}
From this we obtain the condition that $1-y$ must divide $z$, i.e. $z(x,1) \equiv 0$.

\begin{definition}
   The $L^2$ reference polynomial space is given by
   \[
      \widetilde{Z}_N = (1-y) \Qq_{N-1}.
   \]
\end{definition}

\begin{proposition}
   \label{prop:Zh-polyn}
   The local $L^2$ space $\widehat{Z}_N$ contains all polynomials of total degree at most $N-1$.
\end{proposition}
\begin{proof}
   Let $p \in \Pp_{N-1}$, $p = \sum_{\bm\alpha} c_{\bm\alpha} \bm x^{\bm\alpha}$.
   Then,
   \begin{align*}
      (\varphi^*_b p)(x,y)
         &= (1-y) p(x(1-y),y) = \sum_{\bm\alpha} c_{\bm\alpha} x^{\alpha_1}(1-y)^{\alpha_1+1} y^{\alpha_2}.
   \end{align*}
   Since $\alpha_1, \alpha_1 + \alpha_2 \leq N-1$, it follows that $\varphi^*_b p \in \widetilde{Z}_N$.
\end{proof}

\begin{remark}[Lowest-order elements]
   It can be verified that in the lowest-order case ($N=1$), the elements $V(\kk), \bm W(\kk), Z(\kk)$ defined above coincide exactly with the lowest-order Lagrange, Nédélec (first-kind), and discontinuous Lagrange (piecewise constant) elements.
\end{remark}

\subsection{Degrees of freedom}

We construct unisolvent degrees of freedom (dofs) for each of the spaces $V_h$, $\Wh$, $Z_h$ defined above.
The dofs are given in terms of a lattice on the unit square defined by points
\begin{equation}
   \label{eq:lattice-points}
   \begin{gathered}
      0 = x_0 < x_1 < \ldots < x_N = 1, \\
      0 = y_0 < y_1 < \ldots < y_N = 1.
   \end{gathered}
\end{equation}
We introduce a notation to denote the geometric entities of each dimension in this lattice.
The vertex $(x_i, y_j)$ is denoted $(i, j)$, the edge $[x_i, x_{i+1}] \times \{ y_j \}$ is denoted $(\ii, j)$ (and similarly for $(i, \jj)$), and the cell $[x_i, x_{i+1}] \times [y_j, y_{j+1}]$ is denoted $(\ii, \jj)$.

We briefly describe the construction of dofs for the quadrilateral finite element spaces using the entities of this lattice.
The degrees of freedom for the $H^1$ local space $\Qq_N$ are point evaluations at the vertices $(i,j)$.
The associated dual basis functions are the standard Lagrange interpolating polynomials at the points $(x_i, x_j)$.
For the $\Hcurl$ local space $\Qq_{N-1,N} \times \Qq_{N,N-1}$, the degrees of freedom are integrals of the tangential component over the edges $(\ii, j)$ and $(i, \jj)$.
The associated basis functions are the tensor product of a univariate Lagrange interpolating polynomial with a univariate \textit{histopolating} polynomial (a histopolating polynomial is one with prescribed integrals or mean values over specified intervals).
For the $L^2$ local space $\Qq_{N-1}$, the degrees of freedom are integrals over cells $(\ii, \jj)$, and the associated basis functions are the tensor product of two histopolating polynomials.
These constructions were referred to as interpolation--histopolation bases in \cite{Pazner2023};
similar constructions were considered in \cite{Gerritsma2010} under the name ``mimetic bases''.
These bases have the important property that the discrete differential operators $\nabla : H^1(\Square) \to \HcurlDomain{\Square}$ and $\Curl : \HcurlDomain{\Square} \to L^2(\Square)$ are given exactly by the vertex--edge and edge--cell incidence matrices of the lattice, respectively.

To construct degrees of freedom for the local spaces on the unit triangle, we pull back the quadrilateral degrees of freedom under the Duffy transformation.
The lattice on $\Square$ defined by the points $(x_i, y_j)$ is mapped to a lattice on $\Tri$ defined by $(\tilde{x}_i, \tilde{y}_j) = \varphi(x_i, y_j)$.
This lattice consists of $N(N-1)$ quadrilateral cells and $N$ triangular cells.
The mapped degrees of freedom correspond to point evaluations at the vertices, integrals over the edges, and integrals over the cells of this lattice.

Under $\varphi$, all of the vertices $(x_i, 1)$ are mapped to the same point $(0, 1)$, so there are a total $N^2 + N + 1$ vertices in this lattice.
As a consequence, all of the $H^1$ degrees of freedom that are point evaluations along the top edge of $\Square$ collapse to a single degree of freedom, which is point evaluation at the upper vertex of the unit triangle.
All of the edges $[x_i, x_{i+1}] \times \{ 1 \}$ are mapped to the single point $(0,1)$, resulting in $2N^2 + N$ edges in the triangular lattice.
The $\Hcurl$ degrees of freedom corresponding to the top edge of $\Xi$ are all mapped to the zero functional, and can be discarded.
Each subcell in the unit square is mapped to a distinct subcell in the unit triangle, and so there are $N^2$ degrees of freedom in both the mapped and original $L^2$ space.
These degrees of freedom are unisolvent, and the associated dual basis functions can be constructed explicitly.

\begin{theorem}
   \label{thm:unisolvent}
   The following are unisolvent degrees of freedom for the spaces $\widehat{V}_N$, $\widehat{\bm W}_N$, and $\widehat{Z}_N$.
   \begin{enumerate}[label=(\roman*)]
      \item The degrees of freedom for the $H^1$ space $\widehat{V}_N$ are point evaluations,
      \begin{equation}
         \label{eq:h1-dofs}
         v \mapsto v(\tilde{x}_i, \tilde{y}_j),
      \end{equation}
      for all vertices $(i,j)$ in the triangular lattice.
      \item The degrees of freedom for the $\Hcurl$ space $\widehat{\bm W}_N$ are integrals of the tangential components over edges, i.e.
      \begin{equation}
         \label{eq:hcurl-dofs}
         \bm w \mapsto \int_e \bm w \cdot \bm t \, ds,
      \end{equation}
      where $e$ is an edge $(\ii, j)$ or $(i, \jj)$ in the triangular lattice.
      \item The degrees of freedom for the $L^2$ space $\widehat{Z}_N$ are cell integrals,
      \begin{equation}
         \label{eq:l2-dofs}
         z \mapsto \int_\tau z \, d\bm x,
      \end{equation}
      where $\tau$ is a cell $(\ii, \jj)$ in the triangular lattice.
   \end{enumerate}
\end{theorem}
\begin{proof}
   \textbf{Claim (\textit{i}).}
   First note that by \eqref{eq:h1-defn}, $\dim(\widehat{V}_N) = \dim(\widetilde{V}_N) = N^2 + N + 1$, so it suffices to show that any element of $\widehat{V}_N$ that vanishes at all lattice vertices must be identically zero.
   For $v \in \widehat{V}_N$, let $\tilde{v} = \varphi^*_g v \in \widetilde{V}_N$.
   Suppose that $v(\tilde{x}_i, \tilde{y}_j) = 0$ for all vertices $(i,j)$ of the triangular lattice.
   Then, $\tilde{v}(x_i, y_j) = 0$ for all vertices $(i,j)$ of the square lattice.
   Since $\tilde{v} \in \Qq_N$, and point evaluations at the vertices $(i,j)$ are unisolvent degrees of freedom for $\Qq_N$, $\tilde{v} \equiv 0$ and so $v \equiv 0$.
   The basis functions dual to these degrees of freedom are the pullbacks of the Lagrange interpolating polynomials defined on the square lattice.

   \textbf{Claim (\textit{ii}).}
   By \eqref{eq:hcurl-defn}, $\dim(\widehat{\bm W}_N) = \dim(\widetilde{\bm W}_N) = 2N^2 + N$, which is equal to the number of edges in the collapsed lattice.
   We demonstrate unisolvence by constructing the basis dual to the degrees of freedom.
   Let $L_{x,i}(x)$ denote the Lagrange interpolating polynomial satisfying $L_{x,i}(x_j) = \delta_{ij}$.
   Similarly, define $L_{y,i}(y)$ by $L_{y,i}(y_j) = \delta_{ij}$, since, in general, different lattice points may be used along the $x$ and $y$ axes.
   Let $\widehat{L}_{x,i}$ denote the unique degree-$(N-1)$ polynomial satisfying
   \[
      \int_{x_{j}}^{x_{j+1}} \widehat{L}_{x,i}(x) \, dx = \delta_{ij},
   \]
   and similarly for $\widehat{L}_{y,i}$.
   Note that $\widehat{L}_{x,i}(x) = \sum_{k=i+1}^N L'_{x,k}(x)$.
   For the edge $(\ii, j)$, it is straightforward to verify that
   \[
      \widetilde{\bm w}_{\ii,j}(x,y) = \frac{1}{1-y_j}\begin{pmatrix}
         (1-y) \widehat{L}_{x,i}(x) L_{y,j}(y) \\
         0
      \end{pmatrix}
      \in ((1-y)^2 \Qq_{N-1}) \times \{ 0 \} \subseteq \widetilde{\bm W}_N,
   \]
   satisfies the desired property on the square lattice, and so we can set $\bm w_{\ii, j} = \psi^*_c \widetilde{\bm w}_{\ii,j}$.

   On edges $(i,\jj)$ the construction is slightly more involved.
   For $\widetilde{\bm w}_{i,\jj} = (\widetilde{w}_1, \widetilde{w}_2)$, the component $\widetilde{w}_2$ must be given by
   \[
      \widetilde{w}_2(x,y) = L_{x,i}(x) \widehat{L}_{y,j} (y).
   \]
   The condition that $\widetilde{\bm w}_{i,\jj} \in \widetilde{\bm W}_N$ then dictates that $\widetilde{w}_1$ must be given by
   \begin{equation}
      \label{eq:w1-defn}
      \widetilde{w}_1(x,y) = (1-y)^2 s(x,y) - (1-y) L'_{x,i}(x) \widehat{L}_{y,j}(y)
   \end{equation}
   for some $s \in \Qq_{N-1}$.
   The duality condition requires that
   \begin{equation}
      \label{eq:w1-duality}
      \int_{x_k}^{x_{k+1}} \widetilde{w}_1(x,y_\ell) \, dx = 0
   \end{equation}
   for all $k$ and $\ell$.
   Inserting \eqref{eq:w1-defn} into \eqref{eq:w1-duality}, we obtain
   \begin{align*}
      \int_{x_k}^{x_{k+1}} s(x,y_\ell) \, dx
         &= \frac{\widehat{L}_{y,j}(y_\ell)}{1 - y_\ell} L_{x,i}(x)\Big|_{x_k}^{x_{k+1}}
         = \frac{\widehat{L}_{y,j}(y_\ell)}{1 - y_\ell} (\delta_{i,k+1} - \delta_{ik}).
   \end{align*}
   We write $s(x,y) = s_1(x) s_2(y)$, where $s_2(y)$ is the unique degree-$(N-1)$ polynomial that interpolates the function $\widehat{L}_{y,j}(y)/(1-y)$ at the $N$ points $y_\ell$, $\ell < N + 1$.
   Then, $s_1(x)$ must satisfy
   \begin{align*}
      \int_{x_k}^{x_{k+1}} s_1(x) \, dx
         &= \delta_{i,k+1} - \delta_{ik}
         = \delta_{i-1,k} - \delta_{ik},
   \end{align*}
   and hence,
   \[
      s_1(x) = \widehat{L}_{x,{i-1}}(x) - \widehat{L}_{x,i}(x),
   \]
   where the first term is omitted in the case that $i = 1$.

   \textbf{Claim (\textit{iii}).}
   To demonstrate unisolvence of the degrees of freedom, we explicitly construct a dual basis.
   Consider the $(\ii,\jj)$ cell $[x_i, x_{i+1}] \times [y_j, y_{j+1}]$.
   For $j < N$, define $f_j$ to be the degree-$N$ polynomial satisfying
   \[
      f_j(y_i) = \begin{cases}
         -1/(1 - y_i) \quad&\text{ if } i \leq j, \\
         0 \quad&\text{ if } i > j.
      \end{cases}
   \]
   Since $j < N$, we have $f_j(y_{N}) = 0$, and so $1-y$ divides $f_j$, i.e.\ $f_j(y) = (1-y)g_j(y)$, where $\deg(g_j) = N - 1$.
   Set $h_j(y) = (1-y) f_j(y) = (1-y)^2 g_j(y)$ and define $z_j(y) := h_j'(y) = (1-y)(2 g_j(y) + (1-y)g_j'(y))$.
   The integrals over each interval $[y_i, y_{i+1}]$ can be computed,
   \begin{align*}
      \int_{y_i}^{y_{i+1}} z_j(y) \, dy
         &= \int_{y_i}^{y_{i+1}} h_j'(y) \, dy
         = h_j(y_{i+1}) - h_j(y_i)
         = \delta_{ij}.
   \end{align*}
   Define the functions $\tilde{z}_{ij} \in \widetilde{Z}_N$ by
   \[
      \tilde{z}_{ij}(x,y) = \widehat{L}_{x,i}(x) z_j(y),
   \]
   which satisfy
   \begin{align*}
      \int_{y_\ell}^{y_{\ell + 1}} \int_{x_k}^{x_{k+1}} \tilde{z}_{ij} \, dx \, dy
         &= \int_{x_k}^{x_{k+1}} \widehat{L}_{x,i}(x) \, dx \int_{y_\ell}^{y_{\ell + 1}} z_j(y) \, dy
         = \delta_{ik} \delta_{j\ell}.
   \end{align*}
   Setting $z_{ij} = \psi^*_b \tilde{z}_{ij} \in \widehat{Z}_N$, and letting $\tau$ denote the $(\ii, \jj)$ cell of the triangular lattice,
   \begin{align*}
      \int_\tau z_{ij} \, d\bm x
         = \int_{y_\ell}^{y_{\ell + 1}} \int_{x_k}^{x_{k+1}} \tilde{z}_{ij} \, dx \, dy = \delta_{ik} \delta_{j\ell},
   \end{align*}
   so the basis $z_{ij}$ is dual to the degrees of freedom.
\end{proof}

\subsection{Construction of global spaces}

After the specification of the local spaces $V_N(\kk), \bm W(\kk), Z(\kk)$, the global finite element spaces $V_h, \Wh, Z_h$ can be constructed in the standard manner, through the use of global degrees of freedom, which result from the identification of coincident local degrees of freedom on shared geometric entities (vertices and edges in $V_h$, and edges in $\Wh$; the global space $Z_h$ is broken, and no global identification is required).
By \Cref{prop:h1-trace}, the trace of an element of $V_N(\kk)$ on each mesh edge is a polynomial of degree at most $N$, and so identifying coincident degrees of freedom (vertex values and $N - 1$ distinct point values on each edge interior) ensures continuity across elements, and so $V_h \subseteq H^1(\Omega)$.
Similarly, by \Cref{prop:hcurl-trace}, the tangential trace of an element of $\bm W(\kk)$ is a polynomial of degree at most $N - 1$.
The $\Hcurl$ dofs on each edge are integrals of the tangential trace over $N$ disjoint subintervals, which uniquely identifies a polynomial of degree $N-1$ (see e.g.\ \cite{Chihara1978}).
Therefore, identification of coincident dofs ensures continuity of tangential traces across elements, and so $\Wh \subseteq \HcurlOmega$.
This completes the construction of the global spaces defined in (\ref{eq:Vh}--\ref{eq:Zh}).

Additionally, by \Cref{prop:h1-trace}, the element $V(\kk)$ is compatible with the standard $\Qq_N$ and $\Pp_N$ elements on quadrilaterals and triangles, respectively.
Similarly, by \Cref{prop:hcurl-trace}, the element $\bm W(\kk)$ is compatible with the degree-$N$ quadrilateral Nédélec element and first-kind triangular Nédélec elements.
In particular, this means that the spaces (\ref{eq:Vh}--\ref{eq:Zh}) can be defined on mixed meshes containing both triangles and quadrilaterals.
The degrees of freedom (\ref{eq:h1-dofs}--\ref{eq:l2-dofs}) have analogues for quadrilateral elements using the square lattice \eqref{eq:lattice-points} that coincide exactly on edges of the mesh skeleton.

\begin{remark}[Best approximation]
   By \Cref{prop:Vh-polyn,prop:Wh-polyn,prop:Zh-polyn}, the local spaces contain all polynomials of degree $N$ (in the case of $\Vh$) or degree $N-1$ (in the case of $\Wh$ and $\Zh$).
   In fact, the spaces $\Vh$, $\Wh$, and $\Zh$ contain the standard Lagrange, Nédélec, and $L^2$ simplicial finite element spaces, together with certain rational functions.
   Therefore, the best approximation in these spaces is at least as good as in the standard spaces, and the asymptotic convergence rates and error estimates are expected to be comparable to those of the standard spaces.
\end{remark}

\begin{remark}[Quadrature]
   Although the spaces $\Vh, \Wh, \Zh$ contain rational functions, numerical integration can be performed accurately using standard Gaussian (or Jacobi--Gauss) quadrature rules by changing variables to the reference square.
   Under this change of variables, by construction of the reference spaces $\widetilde{V}_N, \widetilde{\bm W}_N, \widetilde{Z}_N$, the integrands are all polynomials.
   This corresponds to the use of a Stroud conical quadrature rule for the triangle \cite{Stroud1971}.
\end{remark}

\begin{remark}[Sum factorization]
   The use of tensor-product Gaussian quadrature rules on the reference square, combined with the tensor-product structure of the lattice \eqref{eq:lattice-points} defining the degrees of freedom, allows for the use of sum factorization techniques to reduce the computational complexity of operator evaluation and matrix assembly with these spaces.
   Using sum factorization, the cost of operator evaluation can be reduced from $\mathcal{O}(N^4)$ to $\mathcal{O}(N^3)$, while also obviating the need to assemble the system matrix, which, in the naive method, requires $\mathcal{O}(N^6)$ operations.
   This motivates the desire for \textit{matrix-free} methods, and in particular, matrix-free preconditioners, which is the focus of the subsequent section.
\end{remark}

\section{Low-order-refined preconditioning}
\label{sec:lor}

Consider the bilinear forms defined on the spaces $\Vh, \Wh, \Zh$,
\begin{align}
   \label{eq:a-h1}
   a_{\Vh}(u, v) &:= (\nabla u, \nabla v) \\
   \label{eq:a-hcurl}
   A_{\Wh}(\bm u, \bm v) &:= (\Curl \bm u, \Curl \bm v) \\
   \label{eq:a-l2}
   a_{Z_h}(u, v) &:= (u, v),
\end{align}
and the associated stiffness matrices $A_{V_h}, A_{\Wh}$ and mass matrix $A_{Z_h}$ in the basis dual to the degrees of freedom (\ref{eq:h1-dofs}--\ref{eq:l2-dofs}).
For \eqref{eq:a-h1} and \eqref{eq:a-hcurl}, we also impose homogeneous Dirichlet boundary conditions on $\partial\Omega$.
The condition numbers of the stiffness matrices scale as $\mathcal{O}(h^{-2} N^4)$, and so effectively preconditioning these systems is essential.
In this section, we develop spectrally equivalent low-order refined preconditioners;
we demonstrate that certain lowest-order discretizations are spectrally equivalent, independent of $N$, to the high-order system, and so any effective preconditioner for the lowest-order preconditioner can be applied directly as a preconditioner for the high-order system.
The advantage of the low-order-refined system is that the associated matrix is very sparse, and there are a large number of existing effective preconditioners for low-order finite element discretizations that can be chosen from.

\subsection{Norm and spectral equivalences}

In the present section, we are concerned with the construction of low-order discretizations on refined meshes that give rise to matrices $A_{V_0}, A_{\bm W_0}, A_{Z_0}$ that are spectrally equivalent to $A_{V_h}, A_{\Wh}, A_{Z_h}$, independent of the polynomial degree $N$.
From $\T$, a refined mesh $\T_0$ is constructed using the collapsed lattice defined by \eqref{eq:lattice-points}.
By superimposing this lattice on each mesh element, each triangle $\kk \in \T$ is decomposed into $N^2 - N$ quadrilaterals and $N$ triangles, resulting in the refined mesh $\T_0$.
On the refined mesh, the standard lowest-order finite element spaces $V_0, \bm W_0, Z_0$ can be formed.
Note that the degrees of freedom for each of the lowest-order spaces on the refined mesh $\T_0$ coincide exactly with the degrees of freedom for the high-order spaces on the original mesh.
Letting $I_V, I_{\bm W}, I_Z$ and $I_{V_0}, I_{\bm W_0}, I_{Z_0}$ denote the nodal interpolation operators defined by the dofs (\ref{eq:h1-dofs}--\ref{eq:l2-dofs}), we obtain the following commutative diagram.
\begin{equation}
   \label{eq:cd}
   \begin{tikzcd}
      H^1(\Omega) \arrow[d, "I_V"] \arrow[r, "\nabla"] &
      \HcurlOmega \arrow[d, "I_{\bm W}"] \arrow[r, "\Curl"] &
      L^2(\Omega) \arrow[d, "I_Z"] \\
      \Vh  \arrow[d, "I_{V_0}"] \arrow[r, "\nabla"] &
      \Wh  \arrow[d, "I_{\bm W_0}"] \arrow[r, "\Curl"] &
      \Zh \arrow[d, "I_{Z_0}"] \\
      V_0 \arrow[r, "\nabla"] &
      \bm W_0 \arrow[r, "\Curl"] &
      Z_0
   \end{tikzcd}
\end{equation}

\begin{remark}[Matrix representations of discrete derivatives and interpolations]
   \label{rem:discrete-ops}
   Because the degrees of freedom for the high-order spaces $\Vh, \Wh, \Zh$ coincide with those of the low-order spaces $V_0, \bm W_0, Z_0$, the matrix representation of each of the interpolation operators $I_{V_0} : \Vh \to V_0$, $I_{\bm W_0} : \Wh \to V_0$, $I_{Z_0} : \Zh \to Z_0$ is the identity matrix.
   Furthermore, by construction of the degrees of freedom in \Cref{thm:unisolvent}, the discrete gradient and curl operators are exactly given by the incidence matrices of the geometric entities of the triangular lattice.
   In particular, since the degrees of freedom for the high-order and low-order matrices coincide, the high-order and low-order discrete differential operators have the same matrix representations.
\end{remark}

\begin{remark}[Sparsity]
   \label{rem:sparsity}
   The element matrices associated with the high-order systems are fully dense: there is an all-to-all coupling between the degrees of freedom belonging to a given element.
   Therefore, total number of nonzeros in the system matrices $A_{\Vh}, A_{\Wh}, A_{\Zh}$ grows as $N^4$, and the number of nonzeros per row grows as $N^2$.
   In contrast, the low-order-refined system matrices have coupling only between neighboring degrees of freedom in the triangular lattice.
   In the $H^1$ space, within a macro-element, away from the top vertex, each degree of freedom is coupled to at most nine degrees of freedom, counting self-connections.
   The top vertex is connected to its $N+1$ adjacent vertices.
   In total, there are $9 N^2 - N + 1$ nonzeros in each low-order-refined macro-element matrix, and so the average number of nonzeros per row of $A_{V_0}$ approaches 9, which is the same as a lowest-order quadrilateral discretization, even though rows associated with collapsed vertex degrees of freedom will have $\mathcal{O}(N)$ nonzeros.
   In the case of the $\Hcurl$ space, there are $\mathcal{O}(1)$ nonzeros per row of $A_{\bm W_0}$.
   The $L^2$ space $Z_0$ is fully decoupled, and so the matrix $A_{Z_0}$ is diagonal.
\end{remark}

Note that while the high-order spaces $\Vh, \Wh, \Zh$ are independent of the choice of lattice points, the refined mesh $\T_0$, and hence the low-order-refined spaces $V_0, \bm W_0, Z_0$ depend on the choice of points.
The spectral equivalence properties of $A_{V_0}, A_{\bm W_0}, A_{Z_0}$ depend strongly on the choice of lattice points.
Our main theorem concerns the choice $x_i = y_i = \frac{1}{2}(1 + \xi^{(0,0)}_{i,N})$, where $\xi^{(0,0)}_{i,N}$ are the Gauss--Lobatto abscissas on the interval $[-1,1]$, discussed in greater detail in \Cref{sec:jacobi}.
In this case, where no ambiguity will arise, we will refer to the points $x_i$ and $y_i$ as the Gauss--Lobatto points on the interval $[0,1]$.

\begin{theorem}
   \label{thm:spectral-equivalence}
   Let $A$ be one of $A_{\Vh}, A_{\Wh}, A_{\Zh}$, and let $A_0 \in \{ A_{V_0}, A_{\bm W_0}, A_{Z_0} \}$ be the corresponding low-order-refined system matrix defined on Gauss--Lobatto lattice points.
   Then, $\kappa(A_0^{-1} A) \approx 1$.
\end{theorem}

The proof of \Cref{thm:spectral-equivalence} can be reduced to certain one-dimensional norm equivalences.
The key equivalences are stated in the following lemma.

\begin{lemma}
   \label{lem:norm-equivalences}
   Consider the $N+1$ Gauss--Lobatto points $\xi_{i,N}^{(0,0)}$.
   Let $(f_0, f_1, \ldots, f_N) \in \mathbb{R}^{N+1}$, and let $u_N$ be the unique polynomial in $\PP^N$ satisfying $u_N(\xi_{i,N}^{(0,0)}) = f_i$.
   Similarly, let $u_h$ be the unique piecewise linear function, defined with respect to the abscissas $\xi_{i,N}^{(0,0)}$, satisfying $u_h(\xi_{i,N}^{(0,0)}) = f_i$.
   Then,
   \begin{align}
      \label{eq:l2-equiv}
      \| u_N \|_0^2 &\approx \| u_h \|_0^2, \\
      \label{eq:h1-equiv}
      \| \partial_x u_N \|_0^2 &\approx \| \partial_x u_h \|_0^2, \\
      \label{eq:weighted-h1-equiv}
      \| (1-x)^{1/2} \partial_x u_N \|_0^2 &\approx \| (1-x)^{1/2} \partial_x u_h \|_0^2.
   \end{align}
   If $u_N$ (and hence $u_h$) vanish at $x=1$, then
   \begin{equation}
      \label{eq:inv-weighted-l2-equiv}
      \| (1-x)^{-1/2} u_N \|_0^2 \approx \| (1-x)^{-1/2} u_h \|_0^2.
   \end{equation}
   If $u_N$ is of degree no more than $N - 1$, then
   \begin{equation}
      \label{eq:weighted-l2-equiv}
      \| (1-x)^{1/2} u_N \|_0^2 \approx \| (1-x)^{1/2} u_h \|_0^2, \qquad u_N \in \PP^{N-1}.
   \end{equation}
\end{lemma}
\begin{proof}
   The equivalences \eqref{eq:l2-equiv} and \eqref{eq:h1-equiv} are well-known in the FEM--SEM literature; see, for example, \cite{Canuto1994,Canuto2010}.
   The weighted $H^1$ seminorm equivalence \eqref{eq:weighted-h1-equiv} is a fundamental new result that allows uniform spectral equivalence bounds on the triangle, and is a consequence of \Cref{thm:h1-interp-equiv}, proven in the subsequent section.
   The weighted $L^2$ equivalences are established in \Cref{lem:inv-weighted-l2} and \Cref{lem:weighted-l2}.
\end{proof}

The $H^1$ seminorm equivalences established above give rise to $L^2$ norm equivalences of high-order and low-order histopolants.
For integrable $f : \Lambda \to \mathbb{R}$, we introduce the notation for the subcell integrals,
\[
   S_i(f) = \int_{\xi_{i,N}^{(0,0)}}^{\xi_{i+1,N}^{(0,0)}} f(x) \, dx.
\]

\begin{corollary}
   \label{cor:norm-equivalences}
   Let $(f_0, f_1, \ldots, f_{N-1}) \in \mathbb{R}^N$, and let $\widehat{u}_{N}$ be the unique polynomial in $\PP^{N-1}$ satisfying $S_i(\widehat{u}_N) = f_i$.
   Let $\widehat{u}_h$ be the unique piecewise constant function, defined with respect to the abscissas $\xi_{i,N}^{(0,0)}$, taking value $h_i^{-1} f_i$ on the interval $[\xi_{i,N}^{(0,0)}, \xi_{i+1,N}^{(0,0)}]$, where $h_i = \xi_{i+1,N}^{(0,0)} - \xi_{i,N}^{(0,0)}$ (i.e.\ so that $S_i(\widehat{u}_h) = f_i$).
   Then,
   \begin{align}
      \label{eq:histop-equiv}
      \| \widehat{u}_N \|_0^2 &\approx \| \widehat{u}_h \|_0^2, \\
      \label{eq:weighted-histop-equiv}
      \| (1-x)^{1/2} \widehat{u}_N \|_0^2 &\approx \| (1-x)^{1/2} \widehat{u}_h \|_0^2.
   \end{align}

   Furthermore, let $\widehat{z}_N$ be the unique polynomial in $(1-x)\mathbb{P}^{N-1}$ satisfying $S_i(\widehat{z}_N) = f_i$, and let $\widehat{z}_h$ be piecewise constant on the first $N-1$ subcells $[\xi_{i,N}^{(0,0)}, \xi_{i+1,N}^{(0,0)}]$, $i < N-1$, and a multiple of $(1-x)$ on the last subcell $[\xi_{N-1,N}^{(0,0)}, \xi_{N,N}^{(0,0)}]$.
   Then,
   \begin{equation}
      \label{eq:inv-weighted-histop-equiv}
      \| (1-x)^{-1/2} \widehat{z}_N \|_0^2 \approx \| (1-x)^{-1/2} \widehat{z}_h \|_0^2.
   \end{equation}
\end{corollary}
\begin{proof}
   Following \cite{Kolev2022}, let $v_N \in \PP^N$ be an antiderivative for $\widehat{u}_N$.
   Then, by \eqref{eq:h1-equiv},
   \[
      \| \widehat{u}_N \|_0^2
         = \| \partial_x v_N \|_0^2
         \approx \sum_{i=0}^{N-1} h_i^{-1} (v_N(\xi_{i+1,N}^{(0,0)}) - v_N(\xi_{i,N}^{(0,0)}))^2
         = \sum_{i=0}^{N-1} h_i^{-1} \left( \int_{\xi_{i,N}^{(0,0)}}^{\xi_{i+1,N}^{(0,0)}} \widehat{u}_N(x) \, dx \right)^2
         = \| \widehat{u}_h \|_0^2,
   \]
   proving \eqref{eq:histop-equiv}.
   Similarly, by \eqref{eq:weighted-h1-equiv},
   \begin{align*}
      \| (1-x)^{1/2} \widehat{u}_N \|_0^2
         = \int_{-1}^1 (1-x) (\partial_x v_N(x))^2 \, dx
         &\approx \sum_{i=0}^{N-1} (v_N(\xi_{i+1,N}^{(0,0)}) - v_N(\xi_{i,N}^{(0,0)}))^2
            \int_{\xi_{i,N}^{(0,0)}}^{\xi_{i+1,N}^{(0,0)}} (1-x) \, dx \\
         &= \sum_{i=0}^{N-1} h_i^{-1} \left( \int_{\xi_{i,N}^{(0,0)}}^{\xi_{i+1,N}^{(0,0)}} \widehat{u}_N(x) \, dx \right)^2
            \left(\int_{\xi_{i,N}^{(0,0)}}^{\xi_{i+1,N}^{(0,0)}} (1-x) \, dx \right) \\
         &= \| (1-x)^{1/2} \widehat{u}_h \|_0^2,
   \end{align*}
   establishing \eqref{eq:weighted-histop-equiv}.
   To prove \eqref{eq:inv-weighted-histop-equiv}, we use the same antiderivative construction, together with the weighted interpolation stability result of \Cref{thm:weighted-interp-stability}.
\end{proof}

Having established these one-dimensional norm equivalences, we proceed to prove \Cref{thm:spectral-equivalence}.

\begin{proof}[Proof of \Cref{thm:spectral-equivalence}]
   We restrict our attention to the reference triangle $\Tri$;
   the extension to general meshes follows from a change of variables argument.
   Let $u_N \in \widehat{V}_N$, and let $u_0 = I_{V_0} u_N$.
   We claim that $\| \nabla u_N \|_0 \approx \| \nabla u_0 \|_0$.
   Note that $u_N = \psi^*_g \widetilde{u}_N$, and define $\bm w_N := \nabla u_N = \psi^*_c (\nabla \widetilde{u}_N) = \psi^*_c \widetilde{\bm w}_N$.
   Let $\widetilde{\bm w}_N = (\widetilde{w}_{1,N}, \widetilde{w}_{2,N})$.
   By commutativity of the diagram \eqref{eq:cd}, $\bm w_0 = \nabla u_0 = I_{\bm W_0} \bm w_N$, with $\bm w_0 = \psi^*_c \widetilde{\bm w}_0 = (\widetilde{w}_{1,0}, \widetilde{w}_{2,0})$.
   Note that $\widetilde{w}_{1,0}$ is piecewise constant in $x$ and piecewise linear in $y$, and $\widetilde{w}_{2,0}$ is piecewise linear in $x$ and piecewise constant in $y$.
   The functions $\widetilde{w}_{1,N}$ and $\widetilde{w}_{1,0}$ have the same integrals over all edges $(\ii, j)$ in the square lattice, and $\widetilde{w}_{2,N}$ and $\widetilde{w}_{2,0}$ have the same integrals over all edges $(i, \jj)$.
   Changing variables as in \eqref{eq:hcurl-integral},
   \begin{align*}
      \| \nabla u_N \|_{0}^2
         &= \int_{\Tri} \| \bm w_N \|^2 \, d\bm x
         = \int_{\Square} \left\|
            \begin{pmatrix}
               1/(1-y) & 0 \\
               x/(1-y) & 1
            \end{pmatrix}
            \begin{pmatrix} \widetilde{w}_{1,N} \\ \widetilde{w}_{2,N}\end{pmatrix}
         \right\|^2 (1-y) \, d\bm x \\
         &= \int_{\Square} \left\|
            \begin{pmatrix}
               1 & 0 \\
               x & 1
            \end{pmatrix}
            \begin{pmatrix} (1-y)^{-1} \widetilde{w}_{1,N} \\ \widetilde{w}_{2,N}\end{pmatrix}
         \right\|^2 (1-y) \, d\bm x.
   \end{align*}
   Note that for $x \in [0,1]$, the matrix $\begin{pmatrix}1 & 0 \\ x & 1\end{pmatrix}$ and its inverse $\begin{pmatrix}1 & 0 \\ -x & 1\end{pmatrix}$ have norm bounded by $(1 + \sqrt{5})/2$,
   and so
   \begin{align*}
      \| \nabla u_N \|_{0}^2
      \approx \int_{\Square} \left\|
            \begin{pmatrix} (1-y)^{-1} \widetilde{w}_{1,N} \\ \widetilde{w}_{2,N}\end{pmatrix}
         \right\|^2 (1-y) \, d\bm x
      = \| (1-y)^{-1/2} \widetilde{w}_{1,N} \|_0^2 + \| (1-y)^{1/2} \widetilde{w}_{2,N} \|_0^2.
   \end{align*}
   Since $\widetilde{\bm w}_N = \nabla \widetilde{u}_N$ with $\widetilde{u}_N \in (1-y) \Qq_{N,N-1} + \Qq_0$, we have $\widetilde{w}_{1,N} \in (1-y)\Qq_{N-1}$, and so, by \eqref{eq:histop-equiv} and \eqref{eq:inv-weighted-l2-equiv},
   \[
      \| (1-y)^{-1/2} \widetilde{w}_{1,N} \|_0^2
      \approx \| (1-y)^{-1/2} \widetilde{w}_{1,0} \|_0^2.
   \]
   Similarly, $\widetilde{w}_{2,N} \in \Qq_{N,N-1}$, so by \eqref{eq:l2-equiv} and \eqref{eq:weighted-histop-equiv},
   \[
      \| (1-y)^{1/2} \widetilde{w}_{2,N} \|_0^2
         \approx \| (1-y)^{1/2} \widetilde{w}_{2,0} \|_0^2,
   \]
   from which it holds that $\| \nabla u_N \|_{0}^2 \approx \| \nabla u_0 \|_{0}^2$, establishing spectral equivalence of $A_{V_h}$ and $A_{V_0}$.

   For $z_N \in Z_N$ and $z_0 = I_{Z_0} z_N$, it holds that $\| z_N \|_0 \approx \| z_0 \|_0$, since $z_N = \psi^*_b \widetilde{z}_N \in (1-y) \Qq_{N-1}$, and by applying \eqref{eq:histop-equiv} in the $x$-dimension, and \eqref{eq:inv-weighted-histop-equiv} in the $y$-dimension.
   This proves spectral equivalence of $A_{Z_h}$ and $A_{Z_0}$.

   To prove spectral equivalence of $A_{\Wh}$ and $A_{\bm W_0}$, we note that, by the discussion in \Cref{rem:discrete-ops}, the high-order and low-order discrete curl operators have the same matrix representation, which we denote here by $C_{\bm W}$.
   Then,
   \[
      A_{\Wh} = C_{\bm W}^T A_{Z_h} C_{\bm W}, \qquad
      A_{\bm W_0} = C_{\bm W}^T A_{Z_0} C_{\bm W}.
   \]
   So, spectral equivalence of $A_{\Wh}$ and $A_{\bm W_0}$ follows directly from the equivalence of $A_{\Zh}$ and $A_{Z_0}$.
\end{proof}

\begin{remark}[Mass matrices]
   The spectral equivalences apply to the stiffness matrices $A_{\Vh}$ and $A_{\Wh}$ and to the $L^2$ mass matrix $A_{\Zh} = M_{\Zh}$.
   However, they \textit{do not apply} to the mass matrices $M_{\Vh}$ and $M_{\Wh}$.
   As will be discussed in \Cref{sec:mass}, the condition number of $M_{V_0}^{-1} M_{\Vh}$ scales as $\mathcal{O}(N)$.
   With different choices of interpolation points, this can be reduced to $\mathcal{O}(1)$, at the expense of symmetry and potential interelement compatibility.
   In the case of the $\Hcurl$ space, $M_{\bm W_0}^{-1}$ is not an effective preconditioner for $M_{\Wh}$.
   This is in stark contrast with the case of tensor-product elements, in which case the mass matrices for all spaces in the high-order finite element de Rham complex are spectrally equivalent to their low-order-refined counterparts uniformly in the polynomial degree \cite{Pazner2023}.
\end{remark}

The remainder of this section is devoted to the Jacobi--Gauss--Lobatto analysis, providing the interpolation stability estimates needed for the one-dimensional norm equivalences of \Cref{lem:norm-equivalences} and \Cref{cor:norm-equivalences}.

\subsection{Jacobi--Gauss--Lobatto analysis}
\label{sec:jacobi}

The norm equivalences of \Cref{lem:norm-equivalences} depend on properties of Gauss--Lobatto and Jacobi--Gauss--Lobatto points and quadratures, and hence on the orthogonal Jacobi polynomials.
In particular, of key importance are polynomial interpolation stability estimates in Jacobi-weighted Sobolev spaces.
The $(\alpha, \beta)$-weighted $L^2$ space on $\Lambda = [-1,1]$ is defined as
\[
   L^2_{(\alpha,\beta)}(\Lambda) := \{ v : \Lambda \to \mathbb{R} :
      \| v \|_{\Lambda,(\alpha, \beta)} < \infty \},
\]
where the $(\alpha, \beta)$-weighted $L^2$ norm is defined by
\[
   \| v \|^2_{\Lambda,(\alpha, \beta)} := \int_\Lambda v(x)^2 (1-x)^\alpha (1+x)^\beta \, dx,
\]
induced by the weighted inner product
\[
   (u ,v)_{\Lambda,(\alpha, \beta)} := \int_\Lambda u(x) v(x) (1-x)^\alpha (1+x)^\beta \, dx.
\]
Where there is no ambiguity, the domain $\Lambda$ will be omitted in the subscript.
The weighted Sobolev spaces are defined as
\[
   H^m_{(\alpha,\beta)}(\Lambda) = \{ u \in L^2_{(\alpha,\beta)}(\Lambda) : \partial_x^k \in L^2_{(\alpha,\beta)}(\Lambda), 0 \leq k \leq m \}.
\]
We will also make use of the non-uniformly (anisotropically) weighted spaces,
\[
   B^m_{(\alpha,\beta)}(\Lambda) = \{ u \in L^2_{(\alpha,\beta)}(\Lambda) : \partial_x^k \in L^2_{(\alpha + k ,\beta + k)}(\Lambda), 0 \leq k \leq m \},
\]
such that the derivatives of different orders are weighted differently.
Let $H^m_{0,(\alpha,\beta)}(\Lambda)$ and $B^m_{0,(\alpha,\beta)}(\Lambda)$ denote the subspaces of $H^m_{(\alpha,\beta)}(\Lambda)$ and $B^m_{(\alpha,\beta)}(\Lambda)$ respectively, consisting of functions that vanish at $\pm 1$.

The Jacobi polynomials $P^{(\alpha, \beta)}_n$ are orthogonal with respect to the $(\alpha, \beta)$-weighted inner product.
Their norms are given by
\begin{equation}
   \label{eq:jacobi-norm}
   \| P_n^{(\alpha,\beta)} \|_{(\alpha,\beta)}^2
      = \frac{2^{\alpha + \beta + 1} \Gamma(n + \alpha + 1) \Gamma(n + \beta + 1)}{(2n + \alpha + \beta + 1) n! \Gamma(n + \alpha + \beta + 1)}.
\end{equation}
The $(N+1)$-point $(\alpha, \beta)$-Jacobi--Gauss--Lobatto quadrature is defined by the abscissas and weights $(\xi_{i,N}^{(\alpha,\beta)}, \rho_{i,N}^{(\alpha,\beta)})_{i=0}^N$.
The abscissas are the roots of the degree-$(N+1)$ polynomial $(1-x^2) (\partial_x P_N^{(\alpha,\beta)})(x)$, and $\rho_{i,N}^{(\alpha,\beta)}$ are the weights such that the quadrature rule is exact for all polynomials of degree $2N - 1$, i.e.\ for any $u \in \PP^{2N-1}$,
\begin{equation}
   \label{eq:lobatto-exact}
   \int_{-1}^1 u(x) (1-x)^\alpha (1+x)^\beta \, dx = \sum_{i=0}^N u(\xi_{i,N}^{(\alpha,\beta)}) \rho_{i,N}^{(\alpha,\beta)}.
\end{equation}
Although the Jacobi--Gauss--Lobatto quadrature is not exact for polynomials of degree $2N$, it does satisfy the bound
\begin{equation}
   \label{eq:lobatto-bound}
   \| u \|_{(\alpha,\beta)}^2 \leq
   \sum_{i=0}^N u(\xi_{i,N}^{(\alpha,\beta)})^2 \rho_{i,N}^{(\alpha,\beta)}
   \leq (2 + N^{-1}(\alpha+\beta+1)) \| u \|_{(\alpha,\beta)}^2.
\end{equation}
for all $u \in \PP^N$, cf.\ \cite[Lemma 3.3]{Shen2011}.

The abscissas $\xi_{i,N}^{(\alpha,\beta)}$ are related to the Jacobi--Gauss--Lobatto angles $\theta_{i,N}^{(\alpha,\beta)}$ by $\xi_{i,N}^{(\alpha,\beta)} = \cos(\theta_{i,N}^{(\alpha,\beta)})$.
In the case of $(0,0)$-weight (Gauss--Lobatto points), the angles satisfy the following estimate of Sündermann \cite{Sundermann1980},
\begin{equation}
   \label{eq:sundermann}
   \theta_{i,N}^{(0,0)} \in \left[ \frac{2i}{2N+1} \pi, \frac{2i+1}{2N+1} \pi \right].
\end{equation}
Szegő \cite[page 353]{Szego1939} provided the following approximations for the interior weights in terms of the angles,
\begin{equation}
   \label{eq:szego-weight}
   \rho_{i,N}^{(\alpha,\beta)}
   \approx N^{-1} 2^{\alpha+\beta+1} \left(
      \sin \tfrac{\theta}{2}
   \right)^{2\alpha + 1}
   \left(
      \cos \tfrac{\theta}{2}
   \right)^{2\beta + 1}
   = N^{-1} 2^{\alpha+\beta+1} \Big(
      \tfrac{1 - \xi_{i,N}^{(\alpha,\beta)}}{2}
   \Big)^{\alpha + 1/2}
   \Big(
      \tfrac{1 + \xi_{i,N}^{(\alpha,\beta)}}{2}
   \Big)^{\beta + 1/2}.
\end{equation}

In addition to the classical Jacobi polynomials, it will be useful to consider the so-called generalized Jacobi polynomials, defined for $\alpha,\beta \in \mathbb{Z}$ by
\[
   J^{(\alpha,\beta)}_n(x) = \begin{cases}
      \begin{alignedat}{3}
         &(1-x)^{-\alpha}(1+x)^{-\beta} J_{n - n_0}^{(-\alpha,-\beta)}(x),
            && \quad n_0 := -\alpha-\beta,
            && \quad \text{if $\alpha,\beta \leq -1$,} \\
         &(1-x)^{-\alpha} J_{n - n_0}^{(-\alpha,\beta)}(x),
            && \quad n_0 := -\alpha,
            && \quad \text{if $\alpha \leq -1$, $\beta \geq 0$,} \\
         &(1+x)^{-\beta} J_{n - n_0}^{(\alpha,-\beta)}(x),
            && \quad n_0 := -\beta,
            && \quad \text{if $\alpha \geq 0$, $\beta \leq -1$,} \\
         &J_{n - n_0}^{(\alpha,\beta)}(x),
            && \quad n_0 := 0,
            && \quad \text{if $\alpha, \beta \geq 0$,}
      \end{alignedat}
   \end{cases}
\]
for $n \geq n_0$ \cite{Guo2006}.
These polynomials form a complete orthogonal system in $L^2_{(\alpha,\beta)}$.
Define $Q_N^{(\alpha,\beta)}$ by
\[
   Q_N^{(\alpha,\beta)} = \operatorname{span} \{ J_{n_0}^{(\alpha,\beta)}, J_{n_0 + 1}^{(\alpha,\beta)}, \ldots, J_{N}^{(\alpha,\beta)} \}.
\]
The generalized Jacobi polynomials satisfy the recurrence relation for the derivatives,
\begin{equation}
   \label{eq:gen-jacobi-deriv}
   \partial_x J_n^{(\alpha,\beta)} = C_n^{(\alpha,\beta)} J_{n-1}^{(\alpha+1,\beta+1)}, \qquad
   C_n^{(\alpha,\beta)} = \begin{cases}
      -2 (n + \alpha + \beta +1), \quad& \alpha,\beta \leq -1, \\
      -n, \quad& \alpha \leq -1,\beta > -1 \ \text{or}\ \alpha > -1, \beta \leq -1, \\
      \frac{1}{2} (n + \alpha + \beta +1), \quad& \alpha,\beta > -1,
   \end{cases}
\end{equation}
and their weighted $L^2$ norms are given by
\begin{equation}
   \label{eq:gen-jacobi-norm}
   \| J_n^{(\alpha,\beta)} \|_{(\alpha,\beta)}^2 = \| J_{n-n_0}^{(|\alpha|,|\beta|)} \|_{(|\alpha|,|\beta|)}^2,
\end{equation}
where the right-hand side is given by \eqref{eq:jacobi-norm}.

Our first result is an inverse inequality on the space $Q_N^{(\alpha,\beta)}$.
\begin{lemma}
   \label{lem:gen-inv-ineq}
   For $\alpha,\beta \in \mathbb{Z}$ and $u \in Q_N^{(\alpha,\beta)}$, it holds that
   \[
      \| \partial_x u \|_{(\alpha+1,\beta+1)} \lesssim N \| u \|_{(\alpha,\beta)}.
   \]
\end{lemma}
\begin{proof}
   Expanding $u$ in terms of the generalized Jacobi polynomials, $u(x) = \sum_{n=n_0}^N u_i J_n^{(\alpha,\beta)}(x)$, and, by \eqref{eq:gen-jacobi-deriv},
   \[
      \partial_x u(x) = \sum_{n=n_0}^N u_n C_n^{(\alpha,\beta)} J_{n-1}^{(\alpha+1,\beta+1)}(x).
   \]
   Therefore,
   \begin{align*}
      \| \partial_x u(x) \|_{(\alpha+1,\beta+1)}^2
         &\lesssim N^2 \sum_{n=n_0}^N u_n^2 \| J_{n-1}^{(\alpha+1,\beta+1)} \|_{(\alpha+1,\beta+1)}^2.
   \end{align*}
   By \eqref{eq:gen-jacobi-norm} and \eqref{eq:jacobi-norm},
   \[
      \| J_{n-1}^{(\alpha+1,\beta+1)} \|_{(\alpha+1,\beta+1)}^2 \approx N^{-1}, \qquad
      \| J_{n}^{(\alpha,\beta)} \|_{(\alpha,\beta)}^2 \approx N^{-1},
   \]
   and so
   \[
      \| \partial_x u(x) \|_{(\alpha+1,\beta+1)}^2
         \lesssim N^2 \sum_{n=n_0}^N u_n^2 \| J_{n}^{(\alpha,\beta)} \|_{(\alpha,\beta)}^2
         = N^2 \| u \|_{(\alpha,\beta)}^2,
   \]
   as desired.
\end{proof}

The following Marcinkiewicz--Zygmund-type estimate is a direct consequence of \eqref{eq:lobatto-exact} and \eqref{eq:lobatto-bound}.
\begin{lemma}
   \label{lem:quadrature-bound}
   Let $\alpha, \beta > -1$, and let $\sigma$ and $\tau$ be non-positive integers.
   Set $M = \lfloor N - (\sigma+\tau)/2 \rfloor$ and
   \begin{alignat*}{4}
      \gamma &= \alpha + \sigma, \qquad&\qquad k &= \lfloor \sigma / 2 \rfloor, \\
      \delta &= \beta + \tau, \qquad&\qquad \ell &= \lfloor \tau / 2 \rfloor.
   \end{alignat*}
   Then, for any $u \in Q^{(k,\ell)}_M$, it holds that
   \begin{equation}
      \label{eq:quadrature-bound}
      \| u \|_{(\gamma,\delta)}^2 \leq \sum_{j=0}^{N} u(\xi_{j,N}^{(\alpha,\beta)})^2 (1-\xi_{j,N}^{(\alpha,\beta)})^{\sigma} (1+\xi_{j,N}^{(\alpha,\beta)})^{\tau} \rho_{j,N}^{(\alpha,\beta)},
   \end{equation}
   where the endpoint values are interpreted by continuous extension.
   Equality holds in \eqref{eq:quadrature-bound} if $\deg(u) < N - (\sigma + \tau) / 2$.
\end{lemma}
\begin{proof}
   Since $u \in Q^{(k,\ell)}_M$, $u(x)^2 (1-x)^\sigma (1+x)^\tau$ is a polynomial of degree at most $2M + \sigma + \tau \leq 2N$.
   Then, by \eqref{eq:lobatto-bound}
   \begin{align*}
      \| u \|_{(\gamma, \delta)}^2
         &= \int u(x)^2 (1-x)^{\gamma-\alpha} (1+x)^{\delta-\beta} (1-x)^\alpha (1+x)^\beta \, dx \\
         &\leq \sum_{j=0}^{N} u(\xi_{j,N}^{(\alpha,\beta)})^2 (1-\xi_{j,N}^{(\alpha,\beta)})^{\gamma-\alpha} (1+\xi_{j,N}^{(\alpha,\beta)})^{\delta-\beta} \rho_{j,N}^{(\alpha,\beta)},
   \end{align*}
   from which the conclusion follows.
   Equality holds if $\deg(u) < N - (\sigma + \tau) / 2$ by exactness of the quadrature in \eqref{eq:lobatto-exact}.
\end{proof}

A main object of study will be the degree-$N$ $(\alpha,\beta)$-Jacobi--Gauss--Lobatto interpolation operator $I_N^{(\alpha,\beta)} : C(\Lambda) \to \PP^N$, defined by
\[
   (I_N^{(\alpha,\beta)} u)(\xi_{i,N}^{(\alpha,\beta)}) = u(\xi_{i,N}^{(\alpha,\beta)}).
\]
We also consider a generalized interpolation operator $I^{(\alpha,\beta)}_{(k,\ell),N}$ for non-positive integers $k$ and $\ell$ defined as follows.
Let $k^* = \max\{-1, k\}$, $\ell^* = \max\{-1, \ell\}$, $\hat{k} = \max\{0, -k-1\}$, $\hat{\ell} = \max\{0, -\ell-1\}$, and define
\begin{gather*}
   I^{(\alpha,\beta)}_{(k,\ell),N} : B^1_{(k^*,\ell^*)} \to Q^{(k,\ell)}_{N+\hat{k}+\hat{\ell}},
\end{gather*}
by
\begin{gather*}
   I^{(\alpha,\beta)}_{(k,\ell),N} u (\xi^{(\alpha,\beta)}_{i.N}) = u (\xi^{(\alpha,\beta)}_{i.N}).
\end{gather*}
Note that $I^{(\alpha,\beta)}_{(0,0),N} = I^{(\alpha,\beta)}_N$.
The following useful result is an analogue of \cite[Theorems 3.1 and 4.1]{Bernardi1992}.
\begin{lemma}
   \label{lem:interp-stability}
   Let $\alpha, \beta > -1$, and $\sigma, \tau \in \{-2, -1, 0\}$, and set $\gamma = \alpha + \sigma$ and $\delta = \beta + \tau$.
   Then, for any $u \in B^1_{(\sigma^*,\tau^*)}$, the interpolant $I^{(\alpha,\beta)}_{(\sigma,\tau),N} u \in Q^{(\sigma,\tau)}_{N+\hat{\sigma}+\hat{\tau}}$ satisfies the bound
   \[
      \| I^{(\alpha,\beta)}_{(\sigma,\tau),N} u \|^2_{(\gamma,\delta)} \lesssim \| u \|^2_{(\gamma,\delta)} + N^{-2} \| \partial_x u \|^2_{(\gamma + 1,\delta + 1)}.
   \]
\end{lemma}
\begin{proof}
   We will prove the case when $\sigma, \tau < 0$.
   The case of $\sigma = 0$ or $\tau = 0$ can be handled by a simplified argument, which does not involve singular weights.

   Note that since $\sigma, \tau \geq -2$, it holds that $N + \hat{\sigma} + \hat{\tau} \leq \lfloor N - (\sigma+\tau)/2 \rfloor$, and so \Cref{lem:quadrature-bound} applies to $I^{(\alpha,\beta)}_{(\sigma,\tau),N} u$, resulting in
   \[
      \| I^{(\alpha,\beta)}_{(\sigma,\tau),N} u \|^2_{(\gamma,\delta)}
         \leq \sum_{j=1}^{N-1} u(\xi_{j,N}^{(\alpha,\beta)})^2 (1-\xi_{j,N}^{(\alpha,\beta)})^{\sigma} (1+\xi_{j,N}^{(\alpha,\beta)})^{\tau} \rho_{j,N}^{(\alpha,\beta)},
   \]
   where the first and last term in the quadrature sum vanish since $(1-x)^{-2\sigma}(1+x)^{-2\tau}$ divides $I^{(\alpha,\beta)}_{(\sigma,\tau),N} u(x)^2$.
   The estimate \eqref{eq:szego-weight} for the weight implies
   \[
      (1-\xi_{j,N}^{(\alpha,\beta)})^{\sigma} (1+\xi_{j,N}^{(\alpha,\beta)})^{\tau} \rho_{j,N}^{(\alpha,\beta)}
      \approx N^{-1} (1-\xi_{j,N}^{(\alpha,\beta)})^{\gamma + 1/2 } (1+\xi_{j,N}^{(\alpha,\beta)})^{\delta + 1/2}
   \]
   which, after making the substitution $v(\theta) = u(\cos \theta)$, gives
   \[
      \| I^{(\alpha,\beta)}_{(\sigma,\tau),N} u \|^2_{(\gamma,\delta)}
         \lesssim N^{-1} \sum_{j=1}^{N-1} v(\theta_{j,N}^{(\alpha,\beta)})^2
         (1 - \cos \theta_{j,N}^{(\alpha,\beta)} )^{\gamma + 1/2} (1 + \cos\theta_{j,N}^{(\alpha,\beta)})^{\delta + 1/2}.
   \]
   The angles $\theta_{j,N}^{(\alpha,\beta)}$ are approximately uniformly distributed: each angle belongs to an interval $J_j$ of length $\sim N^{-1}$, and which intersects only a bounded number of other such intervals; in the Gauss--Lobatto case, see \eqref{eq:sundermann}.
   This implies, using the half-angle identity,
   \[
      \| I^{(\alpha,\beta)}_{(\sigma,\tau),N} u \|^2_{(\gamma,\delta)}
         \lesssim N^{-1} \sum_{j=1}^{N-1} \sup_{\theta \in J_j} \left(
            v(\theta)
            (\sin \tfrac{\theta}{2} )^{\gamma + 1/2} (\cos \tfrac{\theta}{2})^{\delta + 1/2} \right)^2.
   \]
   The Sobolev embedding of $H^1(J_j) \hookrightarrow L^\infty(J_j)$ (cf.\ \cite[Lemma B.4]{Shen2011})
   \[
      \| f \|_{L^\infty(J_j)}^2 \lesssim
         N \| f \|_{L^2(J_j)}^2 + N^{-1} \| f' \|_{L^2(J_j)}^2,
   \]
   gives the bound
   \begin{equation}
      \mathtoolsset{multlined-width=0.75\displaywidth}
      \begin{multlined}
         \label{eq:embedding-bound}
         \| I^{(\alpha,\beta)}_{(\sigma,\tau),N} u \|^2_{(\gamma,\delta)}
            \lesssim \sum_{j=1}^{N-1} \Big(
               \| v(\theta) (\sin \tfrac{\theta}{2} )^{\gamma + 1/2} (\cos \tfrac{\theta}{2})^{\delta + 1/2} \|_{L^2(J_j)}^2 \\
               + N^{-2} \| \partial_\theta (v(\theta) (\sin \tfrac{\theta}{2} )^{\gamma + 1/2} (\cos \tfrac{\theta}{2})^{\delta + 1/2}) \|_{L^2(J_j)}^2 \Big).
      \end{multlined}
   \end{equation}
   To bound the second term on the right-hand side, note
   \begin{equation}
      \label{eq:deriv-computation}
      \mathtoolsset{multlined-width=0.85\displaywidth}
      \begin{multlined}
         \partial_\theta \big(v(\theta) (\sin \tfrac{\theta}{2} )^{\gamma + 1/2} (\cos \tfrac{\theta}{2})^{\delta + 1/2}\big) =
         (\sin \tfrac{\theta}{2} )^{\gamma + 1/2} (\cos \tfrac{\theta}{2})^{\delta + 1/2}
            \Big( \partial_\theta v(\theta) \\
             + v(\theta) \big( \tfrac{1}{2} (\gamma+\tfrac{1}{2}) \cot(\tfrac{\theta}{2}) - \tfrac{1}{2} (\delta+\tfrac{1}{2})\tan(\tfrac{\theta}{2}) \big)
            \Big).
      \end{multlined}
   \end{equation}
   Because of the Jacobi--Gauss--Lobatto angle spacing, $\sin(\tfrac{\theta}{2}) \gtrsim N^{-1}$ and $\cos(\tfrac{\theta}{2}) \gtrsim N^{-1}$ on $J_i$, $1 \leq i \leq N - 1$, and therefore
   \begin{equation}
      \label{eq:sin-bound}
      \left( \tfrac{1}{2} (\gamma+\tfrac{1}{2}) \cot(\tfrac{\theta}{2}) - \tfrac{1}{2} (\delta+\tfrac{1}{2})\tan(\tfrac{\theta}{2}) \right)^2 \lesssim N^2.
   \end{equation}
   Together, \eqref{eq:deriv-computation} and \eqref{eq:sin-bound} give
   \begin{equation}
      \label{eq:deriv-bound}
      \mathtoolsset{multlined-width=0.9\displaywidth}
      \begin{multlined}
         \| \partial_\theta (v(\theta) (\sin \tfrac{\theta}{2} )^{\gamma + 1/2} (\cos \tfrac{\theta}{2})^{\delta + 1/2}) \|_{L^2(J_j)}^2
         \lesssim  \| (\partial_\theta v(\theta)) (\sin \tfrac{\theta}{2} )^{\gamma + 1/2} (\cos \tfrac{\theta}{2})^{\delta + 1/2} \|_{L^2(J_j)}^2 \\
         + N^2  \| v(\theta) (\sin \tfrac{\theta}{2} )^{\gamma + 1/2} (\cos \tfrac{\theta}{2})^{\delta + 1/2} \|_{L^2(J_j)}^2.
      \end{multlined}
   \end{equation}
   From \eqref{eq:embedding-bound} and \eqref{eq:deriv-bound}, it holds
   \begin{align*}
      \| I^{(\alpha,\beta)}_{(\sigma,\tau),N} u \|^2_{(\gamma,\delta)}
         &\lesssim \sum_{j=1}^{N-1} \Big(
            \| v(\theta) (\sin \tfrac{\theta}{2} )^{\gamma + 1/2} (\cos \tfrac{\theta}{2})^{\delta + 1/2} \|_{L^2(J_j)}^2 \\ &\hspace*{2cm} +
            N^{-2} \| (\partial_\theta v(\theta)) (\sin \tfrac{\theta}{2} )^{\gamma + 1/2} (\cos \tfrac{\theta}{2})^{\delta + 1/2}) \|_{L^2(J_j)}^2 \Big) \\
         &\leq \int_0^\pi (v(\theta)^2 + N^{-2}(\partial_\theta v(\theta))^2) (\sin \tfrac{\theta}{2} )^{2\gamma + 1} (\cos \tfrac{\theta}{2})^{2\delta + 1} \, d\theta \\
         &= \int_0^\pi (v(\theta)^2 + N^{-2}(\partial_\theta v(\theta))^2) (1 - \cos\theta)^{\gamma + 1/2} (1 + \cos\theta)^{\delta + 1/2} \, d\theta \\
         &= \int_{-1}^1 \left( u(x)^2 + N^{-2} (\partial_x u(x))^2 (1-x^2) \right) (1-x)^\gamma (1+x)^\delta \, dx,
   \end{align*}
   from which the result follows.
\end{proof}

\subsection{Stability of high-order--low-order interpolation}

Bernardi and Maday showed that the $(\alpha,\alpha)$-Jacobi--Gauss--Lobatto interpolation operator for $-1 < \alpha < 1$ is stable in $H^1_{(\alpha,\alpha)}(\Lambda)$ \cite{Bernardi1992}.
For the purposes of low-order preconditioning on the triangle, the Duffy transformation necessitates stability estimates in shifted norms.
In this context, it is sufficient to restrict the domain of the interpolation operator to low-order functions, i.e.\ certain piecewise polynomials.
Let $V_{h,N}^{(\alpha,\beta)}$ denote the space of functions that are piecewise linear with respect to the Jacobi--Gauss--Lobatto points $\xi_{j,N}^{(\alpha,\beta)}$.
We also use $V_{h,N,0}^{(\alpha,\beta)}$ to denote elements of $V_{h,N}^{(\alpha,\beta)}$ that vanish at $\pm 1$.
The main result of this section concerns stability in the $H^1_{(\alpha,\beta)}$-seminorm of the interpolant $I_N^{(0,0)} : V_{h,N}^{(0,0)} \to \PP^N$, for $\alpha,\beta \in \{ 0, 1 \}$.

Let $\Pi_N^{(\alpha,\beta)}$ denote the $(\alpha,\beta)$-weighted $L^2$-orthogonal projection onto $Q_N^{(\alpha,\beta)}$.
The accuracy properties of $\Pi_N^{(\alpha,\beta)}$ were summarized by \citeauthor{Shen2011} in \cite{Shen2011}.
\begin{lemma}[Theorem 6.1 of \cite{Shen2011}]
   \label{lem:proj-approx}
   For any $\alpha, \beta \in \mathbb{Z}$ and $m \leq N +1$, for $u \in L^2_{(\alpha,\beta)}(\Lambda)$ with $\partial_x^k u \in L^2_{(\alpha+k,\beta+k)}(\Lambda)$ for all $k \leq m \leq N + 1$ and $0 \leq \mu \leq m$,
   \[
      \| \partial_x^\mu( \Pi_N^{(\alpha,\beta)}u - u ) \|_{\alpha + \mu, \beta + \mu}^2
         \lesssim \frac{(N - m + 1)!}{(N - \mu + 1)!} (N + m)^{\mu - m} \| \partial_x^m u \|_{(\alpha + m, \beta + m)}^2.
   \]
\end{lemma}
As a consequence of \Cref{lem:proj-approx}, we can obtain weighted $H^1$ stability of the interpolation operator $I^{(0,0)}_{(-2,-1),N}$, which is of use in the analysis of the Duffy-weighted $L^2$-conforming space.
\begin{theorem}
   \label{thm:weighted-interp-stability}
   For $u \in B^1_{(-2,-1)}$, it holds that
   \begin{equation}
      \label{eq:weighted-interp-stability}
      \| \partial_x I^{(0,0)}_{(-2,-1),N} u \|_{(-1,0)} \lesssim \| \partial_x u \|_{(-1,0)}.
   \end{equation}
\end{theorem}
\begin{proof}
   First, note that
   \begin{align*}
       \| \partial_x (I^{(0,0)}_{(-2,-1),N} u - u) \|_{(-1,0)}
         &\leq \| \partial_x (I^{(0,0)}_{(-2,-1),N} u - \Pi^{(-2,-1)}_N u) \|_{(-1,0)}
            + \| \partial_x (\Pi^{(-2,-1)}_N u - u) \|_{(-1,0)} \\
         &\lesssim \| \partial_x (I^{(0,0)}_{(-2,-1),N} u - \Pi^{(-2,-1)}_N u) \|_{(-1,0)}
            + \| \partial_x u  \|_{(-1,0)},
   \end{align*}
   where the last step follows from \Cref{lem:proj-approx}.
   Note that $\partial_x (I^{(0,0)}_{(-2,-1),N} u - \Pi^{(-2,-1)}_N u) \in Q^{(-2,-1)}_N$, and so, by the inverse inequality \Cref{lem:gen-inv-ineq},
   \begin{align*}
      \| \partial_x (I^{(0,0)}_{(-2,-1),N} u - \Pi^{(-2,-1)}_N u) \|_{(-1,0)}
      &\lesssim N \| I^{(0,0)}_{(-2,-1),N} u - \Pi^{(-2,-1)}_N u \|_{(-2,-1)} \\
      &= N \| I^{(0,0)}_{(-2,-1),N} (u - \Pi^{(-2,-1)}_N u) \|_{(-2,-1)}.
   \end{align*}
   By \Cref{lem:interp-stability},
   \[
      \| I^{(0,0)}_{(-2,-1),N} (u - \Pi^{(-2,-1)}_N u) \|_{(-2,-1)}
      \lesssim \| u - \Pi^{(-2,-1)}_N u \|_{(-2,-1)} + N^{-1} \| \partial_x (u - \Pi^{(-2,-1)}_N u) \|_{(-1,0)},
   \]
   and so, by another application of \Cref{lem:proj-approx},
   \[
      \| \partial_x (I^{(0,0)}_{(-2,-1),N} u - u) \|_{(-1,0)}
         \lesssim \| \partial_x \|_{(-1,0)},
   \]
   from which \eqref{eq:weighted-interp-stability} follows by the triangle inequality.
\end{proof}
We now turn to stability estimates of the Gauss--Lobatto interpolation operator $I^{(0,0)}_N$ in shifted Jacobi norms.
We first establish some technical results concerning the projection operators.
The following result gives a bound for the endpoint behavior of the orthogonal projection.
\begin{lemma}
   \label{lem:endpoint-bound}
   Let $u_h \in V_{h,N}^{(0,0)}$ with $u_h(\pm 1) = 0$.
   Then, for $\alpha, \beta \in \{ 0, 1 \}$,
   \[
      | \Pi_N^{(\alpha - 1, \beta - 1)} u_h(\pm 1) | \lesssim \| \partial_x u_h \|_{(\alpha,\beta)}.
   \]
\end{lemma}
\begin{proof}
   We consider $x = -1$; the case of $x = 1$ follows by an analogous argument.
   If $\beta = 0$, then $\Pi_N^{(\alpha - 1, \beta - 1)} u_h(-1) = 0$ by construction, and so we consider only $\beta = 1$.
   Define the interval $J = [\xi_{0,N}^{(0,0)}, \xi_{1,N}^{(0,0)}]$.
   Note that the width of this interval is $h \approx N^{-2}$.
   By the embedding of $H^1$ into $L^\infty$ \cite[Lemma B.4]{Shen2011},
   \begin{equation}
      \label{eq:endpoint-embed}
      \| \Pi_N^{(\alpha - 1, \beta - 1)} u_h - u_h \|_{L^\infty(J)}^2 \lesssim
         N^2 \| \Pi_N^{(\alpha - 1, \beta - 1)} u_h - u_h \|_{L^2(J)}^2 + N^{-2} \| \partial( \Pi_N^{(\alpha - 1, \beta - 1)} u_h - u_h ) \|_{L^2(J)}^2.
   \end{equation}
   The first term on the right-hand side can be estimated, since $(1 - x)^{\alpha - 1} \geq 1$ pointwise,
   \begin{align*}
      \| \Pi_N^{(\alpha - 1, \beta - 1)} u_h - u_h \|_{L^2(J)}^2
         &= \int_{\xi_{0,N}^{(0,0)}}^{\xi_{1,N}^{(0,0)}} (\Pi_N^{(\alpha - 1, \beta - 1)} u_h(x) - u_h(x))^2 \, dx \\
         &\lesssim \int_{\xi_{0,N}^{(0,0)}}^{\xi_{1,N}^{(0,0)}} (\Pi_N^{(\alpha - 1, \beta - 1)} u_h(x) - u_h(x))^2 (1-x)^{\alpha - 1} \, dx \\
         &= \| \Pi_N^{(\alpha - 1, \beta - 1)} u_h - u_h \|_{J,(\alpha - 1,\beta - 1)}^2 \\
         &\leq \| \Pi_N^{(\alpha - 1, \beta - 1)} u_h - u_h \|_{(\alpha - 1,\beta - 1)}^2
   \end{align*}
   and so by \Cref{lem:proj-approx},
   \begin{equation}
      \label{eq:l2-shift-inequality}
      \| \Pi_N^{(\alpha - 1, \beta - 1)} u_h - u_h \|_{L^2(J)}^2
         \lesssim N^{-2} \| \partial_x u_h \|_{(\alpha,\beta)}^2,
   \end{equation}
   where the last step follows from \Cref{lem:proj-approx}.
   Since $u_h$ is piecewise linear with respect to the points $\xi_{j,N}^{(0,0)}$, it is linear on $J$, and so $\Pi_N^{(\alpha - 1, \beta - 1)} u_h - u_h$ restricted to $J$ is a polynomial of degree at most $N$.
   Therefore, by the polynomial inverse inequality \cite[Lemma 2.1]{Canuto1982},
   \begin{align*}
      \| \partial_x (\Pi_N^{(\alpha - 1, \beta - 1)} u_h - u_h) \|_{L^2(J)}^2
         &\lesssim N^4 \| \Pi_N^{(\alpha - 1, \beta - 1)} u_h - u_h \|_{L^2(J)}^2 \\
         &\lesssim N^2 \| \partial_x u_h \|_{(\alpha,\beta)}^2,
   \end{align*}
   using again \eqref{eq:l2-shift-inequality}.
   From \eqref{eq:endpoint-embed},
   \[
      \| \Pi_N^{(\alpha - 1, \beta - 1)} u_h - u_h \|_{L^\infty(J)}^2 \lesssim \| \partial_x u_h \|_{(\alpha,\beta)}^2,
   \]
   from which the endpoint bounds follow.
\end{proof}

We will also consider the projection $\Pi_{N,0}^{(\alpha,\beta)}$ onto the space $\PP^N_0$ of polynomials vanishing at $\pm 1$.
This projection satisfies
\begin{equation}
   \label{eq:def-zero-proj}
   (\Pi_{N,0}^{(\alpha,\beta)} u, v)_{(\alpha,\beta)} = (u, v)_{(\alpha,\beta)} \qquad\text{for all $v \in \PP^N_0$}.
\end{equation}
It can be seen that this projection can be written equivalently as
\[
   \Pi_{N,0}^{(\alpha,\beta)} u = (1-x^2) \Pi_{N-2}^{(\alpha+2,\beta+2)} (u / (1-x^2)).
\]
The following representation, relating $\Pi_{N,0}^{(\alpha,\beta)}$ to the unconstrained projection $\Pi_{N}^{(\alpha,\beta)}$, will be useful for the purposes of the stability analysis.

\begin{lemma}
   \label{lem:proj-correction}
   The projection $\Pi_{N,0}^{(\alpha,\beta)}$ is given by
   \begin{equation}
      \label{eq:proj-correction}
      \Pi_{N,0}^{(\alpha,\beta)} u = \Pi_{N}^{(\alpha,\beta)} u + c_1 J_{N-1}^{(\alpha+1,\beta+1)} + c_2 J_N^{(\alpha+1,\beta+1)},
   \end{equation}
   where $c_1$ and $c_2$ satisfy the linear equations
   \begin{align}
      \label{eq:lin-eqn-1}
      c_1 (-1)^{N-1} \binom{N + \beta}{N-1} + c_2 (-1)^{N} \binom{N + \beta + 1}{N} &= -(\Pi_{N}^{(\alpha,\beta)} u)(-1), \\
      \label{eq:lin-eqn-2}
      c_1 \binom{N + \alpha}{N-1} + c_2 \binom{N + \alpha + 1}{N} &= -(\Pi_{N}^{(\alpha,\beta)} u)(1).
   \end{align}
\end{lemma}
\begin{proof}
   Let $u^*$ be defined by the right-hand side of \eqref{eq:proj-correction}, $u^* = \Pi_{N}^{(\alpha,\beta)} u + c_1 J_{N-1}^{(\alpha+1,\beta+1)} + c_2 J_N^{(\alpha+1,\beta+1)}$.
   For any test function $v \in \PP^N_0$, $v = (1-x^2) w$, $w \in \PP^{N-2}$, and so
   \begin{align*}
      (u^*, v)_{(\alpha,\beta)}
         &= (\Pi_{N}^{(\alpha,\beta)} u, v)_{(\alpha,\beta)} + c_1 (J_{N-1}^{(\alpha+1,\beta+1)}, v)_{(\alpha,\beta)} + c_2 (J_N^{(\alpha+1,\beta+1)}, v)_{(\alpha,\beta)} \\
         &= (u, v)_{(\alpha,\beta)} + c_1 (J_{N-1}^{(\alpha+1,\beta+1)}, w)_{(\alpha+1,\beta+1)} + c_2 (J_N^{(\alpha+1,\beta+1)}, w)_{(\alpha+1,\beta+1)} \\
         &= (u, v)_{(\alpha,\beta)}
   \end{align*}
   since $J_{N-1}^{(\alpha+1,\beta+1)}$ and $J_{N}^{(\alpha+1,\beta+1)}$ are orthogonal to $w \in \PP^{N-2}$ in the $(\alpha+1,\beta+1)$-weighted inner product, so $u^*$ satisfies \eqref{eq:def-zero-proj}.
   The Jacobi polynomials are normalized such that $P_n^{(\alpha,\beta)}(1) = \binom{n+\alpha}{n}$ and $P_n^{(\alpha,\beta)}(-1) = (-1)^n \binom{n+\beta}{n}$ \cite{Shen2011}, so equations \eqref{eq:lin-eqn-1} and \eqref{eq:lin-eqn-2} guarantee $u^*(\pm 1) = 0$.
\end{proof}

Given the representation \eqref{eq:proj-correction}, it will be useful to estimate shifted norms of Jacobi polynomials.
\begin{lemma}
   \label{lem:jacobi-shifted}
   For $\alpha,\beta > 0$, it holds that
   \begin{equation}
      \label{eq:jacobi-shift-l2}
      \| P_{n}^{(\alpha,\beta)} \|_{(\alpha - 1,\beta - 1)}^2 \lesssim 2^{\alpha-1} + 2^{\beta-1}.
   \end{equation}
\end{lemma}
\begin{proof}
   The pointwise bounds of Muckenhoupt give \cite[(2.6), (2.7)]{Muckenhoupt1986}
   \begin{equation}
      \label{eq:muckenhoupt}
      |P_{n}^{(\alpha,\beta)}| \lesssim
      \begin{cases}
         (n+1)^\alpha, &\quad 1 - (n+1)^{-2} \leq x \leq 1,\\
         (n+1)^{-1/2}(1 - x)^{-\alpha/2 - 1/4}, &\quad 0 \leq x \leq 1 - (n+1)^{-2},\\
         (n+1)^{-1/2}(1 + x)^{-\beta/2 - 1/4}, &\quad -1 + (n+1)^{-2} \leq x \leq 0,\\
         (n+1)^\beta, &\quad -1 \leq x \leq -1 + (n + 1)^{-2}.
      \end{cases}
   \end{equation}
   Splitting the integral according to \eqref{eq:muckenhoupt}, we see
   \begin{align*}
      \int_{-1}^{-1 + (n+1)^{-2}} P_{n}^{(\alpha,\beta)}(x)^2 (1-x)^{\alpha-1}(1+x)^{\beta-1} \, dx
         &\lesssim n^{2\beta} \int_{-1}^{-1 + (n+1)^{-2}} (1-x)^{\alpha-1}(1+x)^{\beta-1} \, dx \\
         &\leq n^{2\beta} \int_{-1}^{-1 + (n+1)^{-2}} 2^{\alpha-1} (n^{-2})^{\beta-1} \, dx \\
         &\lesssim n^{2\beta} n^{-2} 2^{\alpha-1} n^{-2\beta + 2} = 2^{\alpha-1}.
   \end{align*}
   The same argument for the right endpoint interval gives
   \[
      \int_{1 - (n+1)^{-2}}^1 P_{n}^{(\alpha,\beta)}(x)^2 (1-x)^{\alpha-1}(1+x)^{\beta-1} \, dx
      \lesssim 2^{\beta-1}.
   \]
   In the interior
   \begin{align*}
      \int_{-1+(n+1)^{-2}}^0 P_{n}^{(\alpha,\beta)}(x)^2 (1-x)^{\alpha-1}(1+x)^{\beta-1} \, dx
         &\lesssim \int_{-1+(n+1)^{-2}}^0 N^{-1} (1-x)^{\alpha-1}(1+x)^{-3/2} \, dx \\
         &\leq N^{-1} 2^{\alpha - 1} \int_{-1+(n+1)^{-2}}^0 (1+x)^{-3/2} \, dx
         \lesssim 2^{\alpha - 1}.
   \end{align*}
   The same argument gives
   \[
      \int_{0}^{1-(n+1)^{-2}} P_{n}^{(\alpha,\beta)}(x)^2 (1-x)^{\alpha-1}(1+x)^{\beta-1} \, dx
         \lesssim 2^{\beta - 1}
   \]
   from which \eqref{eq:jacobi-shift-l2} follows.
\end{proof}

Given these results, it can be seen that the zero-boundary projection $\Pi_{N,0}^{(\alpha,\beta)}$ has similar accuracy properties to the unconstrained projection $\Pi_{N}^{(\alpha,\beta)}$ when restricted to $V_{h,N}^{(0,0)}$.

\begin{lemma}
   \label{lem:bdr-proj-accuracy}
   Let $\alpha,\beta \in \{0, 1\}$ and $u_h \in V_{h,N,0}^{(0,0)}$.
   Then,
   \begin{equation}
      \label{eq:proj-l2-error}
      \| \Pi_{N,0}^{(\alpha-1,\beta-1)} u_h - u_h \|_{(\alpha-1,\beta-1)}
         \lesssim N^{-1} \| \partial_x \|_{(\alpha,\beta)}
   \end{equation}
   and
   \begin{equation}
      \label{eq:proj-h1-error}
      \| \partial_x (\Pi_{N,0}^{(\alpha-1,\beta-1)} u_h - u_h) \|_{(\alpha,\beta)} \lesssim \| \partial_x \|_{(\alpha,\beta)}.
   \end{equation}
\end{lemma}
\begin{proof}
   By \Cref{lem:proj-correction},
   \begin{align*}
      \| \Pi_{N,0}^{(\alpha-1,\beta-1)} u_h - u_h \|_{(\alpha-1,\beta-1)}
         &= \| \Pi_{N}^{(\alpha-1,\beta-1)} u_h + c_1 J_{N-1}^{(\alpha,\beta)} + c_2 J_{N}^{(\alpha,\beta)} - u_h \|_{(\alpha-1,\beta-1)} \\
         &\leq \| \Pi_{N}^{(\alpha-1,\beta-1)} u_h - u_h \|_{(\alpha-1,\beta-1)} + c_1 \| J_{N-1}^{(\alpha,\beta)} \|_{(\alpha-1,\beta-1)} + c_2 \| J_{N}^{(\alpha,\beta)} \|_{(\alpha-1,\beta-1)} \\
         &\lesssim N^{-1} \| \partial_x u_h \|_{(\alpha,\beta)} + c_1 \| J_{N-1}^{(\alpha,\beta)} \|_{(\alpha-1,\beta-1)} + c_2 \| J_{N}^{(\alpha,\beta)} \|_{(\alpha-1,\beta-1)},
   \end{align*}
   where the last step follows from \Cref{lem:proj-approx}.
   It remains only to bound the last two terms on the right-hand side.
   By \Cref{lem:endpoint-bound}, the right-hand sides of \eqref{eq:lin-eqn-1} and \eqref{eq:lin-eqn-2} defining $c_1$ and $c_2$ are both bounded by $\| \partial_x u_h \|_{(\alpha,\beta)}$.
   From this, we obtain $c_1, c_2 \lesssim N^{-1} \| \partial_x u_h \|_{(\alpha,\beta)}$.
   (If $\alpha = \beta = 0$, then $c_1 = c_2 = 0$).
   From \Cref{lem:jacobi-shifted}, we conclude
   \begin{equation}
      \label{eq:c1-c2-bound}
      c_1 \| J_{N-1}^{(\alpha,\beta)} \|_{(\alpha-1,\beta-1)} + c_2 \| J_{N}^{(\alpha,\beta)} \|_{(\alpha-1,\beta-1)} \lesssim N^{-1} \| \partial_x u_h \|_{(\alpha,\beta)}.
   \end{equation}
   proving \eqref{eq:proj-l2-error}.

   To show \eqref{eq:proj-h1-error}, note
   \begin{align*}
      \| \partial_x (\Pi_{N,0}^{(\alpha-1,\beta-1)} u_h - u_h) \|_{(\alpha,\beta)}
         &\leq \| \partial_x ( \Pi_{N}^{(\alpha-1,\beta-1)} u_h - u_h ) \|_{(\alpha,\beta)} + c_1 \| \partial_x J_{N-1}^{(\alpha,\beta)} \|_{(\alpha,\beta)} + c_2 \| \partial_x J_{N}^{(\alpha,\beta)} \|_{(\alpha,\beta)} \\
         &\lesssim \| \partial_x u_h \|_{(\alpha,\beta)} + c_1 \| \partial_x J_{N-1}^{(\alpha,\beta)} \|_{(\alpha,\beta)} + c_2 \| \partial_x J_{N}^{(\alpha,\beta)} \|_{(\alpha,\beta)}.
   \end{align*}
   The inverse inequality for Jacobi polynomials gives $\| \partial_x J_{N-1}^{(\alpha,\beta)} \|_{(\alpha,\beta)} \lesssim N \| J_{N-1}^{(\alpha,\beta)} \|_{(\alpha-1,\beta-1)}$, and similarly for $\| \partial_x J_N^{(\alpha,\beta)} \|_{(\alpha,\beta)}$.
   Then, \eqref{eq:proj-h1-error} holds by virtue of \eqref{eq:c1-c2-bound}.
\end{proof}

\begin{theorem}
   \label{thm:interp-stability}
   The Gauss--Lobatto interpolant $I_N^{(0,0)} : V_{h,N}^{(0,0)} \to \PP^N$ is stable in the $(\alpha, \beta)$-weighted $H^1$-seminorm for $\alpha, \beta \in \{ 0, 1 \}$.
   That is, for all $u_h \in V_{h,N}^{(0,0)}$,
   \begin{equation}
      \label{eq:interp-stability}
      \| \partial_x I_N^{(0,0)} u_h \|_{(\alpha,\beta)} \lesssim \| \partial_x u_h \|_{(\alpha,\beta)}.
   \end{equation}
\end{theorem}
\begin{proof}
   Without loss of generality, we can assume that $u_h$ vanishes at $\pm 1$ (if not, replace $u_h$ with $\tilde{u}_h := u_h - u^*$, where $u^*$ is the linear interpolant of $u$ at the endpoints; note that $\tilde{u}_h \in V_{h,N,0}^{(0,0)}$).
   Then,
   \begin{equation}
      \label{eq:interp-triangle}
      \| \partial_x (I_N^{(0,0)} u_h - u_h) \|_{(\alpha,\beta)}
         \leq \| \partial_x (I_N^{(0,0)} u_h - \Pi_{N,0}^{(\alpha-1,\beta-1)} u_h) \|_{(\alpha,\beta)}
         + \| \partial_x (\Pi_{N,0}^{(\alpha-1,\beta-1)} u_h - u_h) \|_{(\alpha,\beta)}.
   \end{equation}
   To bound the first term on the right-hand side, by the inverse inequality \cite[Theorem 3.32]{Shen2011},
   \begin{align*}
      \| \partial_x (I_N^{(0,0)} u_h - \Pi_{N,0}^{(\alpha-1,\beta-1)} u_h) \|_{(\alpha,\beta)}^2
         &\lesssim N^2 \| I_N^{(0,0)} u_h - \Pi_{N,0}^{(\alpha-1,\beta-1)} u_h \|_{(\alpha - 1,\beta - 1)}^2 \\
         &= N^2 \| I_N^{(0,0)} (u_h - \Pi_{N,0}^{(\alpha-1,\beta-1)} u_h) \|_{(\alpha - 1,\beta - 1)}^2.
   \end{align*}
   Noting that $\alpha - 1 \in \{ \alpha, \alpha -1 \}$ and $\beta - 1 \in \{ \beta, \beta - 1\}$, \Cref{lem:interp-stability} implies
   \[
      N^2 \| I_N^{(0,0)} (u_h - \Pi_{N,0}^{(\alpha-1,\beta-1)} u_h) \|_{(\alpha - 1,\beta - 1)}^2
         \lesssim N^2 \| u_h - \Pi_{N,0}^{(\alpha-1,\beta-1)} u_h \|_{(\alpha - 1,\beta - 1)}^2 + \| \partial_x (u_h - \Pi_{N,0}^{(\alpha-1,\beta-1)} u_h) \|_{(\alpha,\beta)}^2.
   \]
   Then, from \eqref{eq:interp-triangle},
   \[
      \| \partial_x (I_N^{(0,0)} u_h - u_h) \|_{(\alpha,\beta)}
         \lesssim N^2 \| u_h - \Pi_{N,0}^{(\alpha-1,\beta-1)} u_h \|_{(\alpha - 1,\beta - 1)}^2 + \| \partial_x (u_h - \Pi_{N,0}^{(\alpha-1,\beta-1)} u_h) \|_{(\alpha,\beta)}^2.
   \]
   The desired result follows from \Cref{lem:bdr-proj-accuracy}.
\end{proof}

\begin{remark}
   The case of $-1 < \alpha = \beta < 1$ was studied by Bernardi and Maday in the early 1990s \cite{Maday1991,Bernardi1992,Bernardi1997}.
   The same result can also be proven using generalized Jacobi polynomials, cf.\ \cite{Ben-yu2001,Shen2011}.
\end{remark}

\begin{theorem}
   \label{thm:h1-interp-equiv}
   For $u_h \in V_{h,N}^{(0,0)}$ and $\alpha,\beta \in \{ 0, 1 \}$, it holds that
   \[
      \| \partial_x I_N^{(0,0)} u_h \|_{(\alpha,\beta)} \approx \| \partial_x u_h \|_{(\alpha,\beta)}.
   \]
\end{theorem}
\begin{proof}
   In light of the stability result \eqref{eq:interp-stability}, it remains to prove the bound
   \begin{equation}
      \label{eq:pw-lin-interp-stability}
      \| \partial_x u_h \|_{(\alpha,\beta)} \lesssim \| \partial_x I_N^{(0,0)} u_h \|_{(\alpha,\beta)}.
   \end{equation}
   Denote $u_N = I_N^{(0,0)} u_h$.
   Since $(\partial_x u_N)^2 (1-x)^\alpha (1+x)^\beta \in \PP^{2N}$, it is integrated exactly by the Gauss--Legendre quadrature $(\zeta_{i,N}^{(0,0)}, \omega_{i,N}^{(0,0)})$, and
   \begin{equation}
      \label{eq:gauss-quad-est}
      \begin{aligned}
      \| \partial_x u_N \|_{(\alpha,\beta)}^2 =
      \int_{-1}^1 (\partial_x u_N(x))^2 (1-x)^\alpha (1+x)^\beta \, dx
         &= \sum_{i=0}^N  (\partial_x u_N(\zeta_{i,N}^{(0,0)}))^2 (1-\zeta_{i,N}^{(0,0)})^\alpha (1+\zeta_{i,N}^{(0,0)})^\beta \omega_{i,N}^{(0,0)} \\
         &\gtrsim N^{-2} \sum_{i=0}^N  (\partial_x u_N(\zeta_{i,N}^{(0,0)}))^2 \omega_{i,N}^{(0,0)} = N^{-2} \| \partial_x u_N \|_{(0,0)}^2.
      \end{aligned}
   \end{equation}
   We then compute, letting $h_i = \xi_{i+1,N}^{(0,0)} - \xi_{i,N}^{(0,0)}$,
   \begin{align*}
      \| \partial_x u_h \|_{(\alpha,\beta)}^2
         &= \int_{-1}^1 (\partial_x u_h(x))^2 (1-x)^\alpha (1+x)^\beta \, dx \\
         &= \sum_{i=0}^{N-1} \left(\frac{u_h(\xi_{i+1,N}^{(0,0)}) - u_h(\xi_{i,N}^{(0,0)})}{h_i}\right)^2 \int_{\xi_{i,N}^{(0,0)}}^{\xi_{i+1,N}^{(0,0)}} (1-x)^\alpha (1+x)^\beta \, dx \\
         &= \sum_{i=0}^{N-1} h_i^{-2} \left( \int_{\xi_{i,N}^{(0,0)}}^{\xi_{i+1,N}^{(0,0)}} \partial_x u_N(x) \, dx \right)^2 \int_{\xi_{i,N}^{(0,0)}}^{\xi_{i+1,N}^{(0,0)}} (1-x)^\alpha (1+x)^\beta \, dx \\
         &\leq \sum_{i=0}^{N-1} h_i^{-1} \int_{\xi_{i,N}^{(0,0)}}^{\xi_{i+1,N}^{(0,0)}} (\partial_x u_N(x))^2 \, dx \int_{\xi_{i,N}^{(0,0)}}^{\xi_{i+1,N}^{(0,0)}} (1-x)^\alpha (1+x)^\beta \, dx,
   \end{align*}
   where the last step follows from the Cauchy--Schwarz inequality.
   On all but the first and the last interval,
   \[
      h_i^{-1} \int_{\xi_{i,N}^{(0,0)}}^{\xi_{i+1,N}^{(0,0)}} (1-x)^\alpha (1+x)^\beta \, dx
         \lesssim \min_{x \in [\xi_{i,N}^{(0,0)}, {\xi_{i+1,N}^{(0,0)}}]} (1-x)^\alpha (1+x)^\beta,
   \]
   and so
   \begin{align*}
      &\sum_{i=1}^{N-2} h_i^{-1} \int_{\xi_{i,N}^{(0,0)}}^{\xi_{i+1,N}^{(0,0)}} (\partial_x u_N(x))^2 \, dx \int_{\xi_{i,N}^{(0,0)}}^{\xi_{i+1,N}^{(0,0)}} (1-x)^\alpha (1+x)^\beta \, dx \\
         & \hspace{5cm} \lesssim \sum_{i=1}^{N-2} \int_{\xi_{i,N}^{(0,0)}}^{\xi_{i+1,N}^{(0,0)}} (\partial_x u_N(x))^2 \, dx \min_{x \in [\xi_{i,N}^{(0,0)}, {\xi_{i+1,N}^{(0,0)}}]} (1-x)^\alpha (1+x)^\beta \\
         & \hspace{5cm} \leq \sum_{i=1}^{N-2}  \int_{\xi_{i,N}^{(0,0)}}^{\xi_{i+1,N}^{(0,0)}} (\partial_x u_N(x))^2 (1-x)^\alpha (1+x)^\beta \, dx
         \leq \| \partial_x u_N \|_{(\alpha,\beta)}^2.
   \end{align*}
   It remains to bound the terms for $i=0$ and $i=N-1$.
   Considering the first interval, if $\beta = 0$, the bound is immediate, so we can take $\beta = 1$.
   Note that
   \[
      h_0^{-1} \int_{\xi_{0,N}^{(0,0)}}^{\xi_{1,N}^{(0,0)}} (1-x)^\alpha (1+x)^\beta \, dx
         \lesssim N^{-2\beta},
   \]
   and so by \eqref{eq:gauss-quad-est},
   \begin{align*}
      h_0^{-1} \int_{\xi_{0,N}^{(0,0)}}^{\xi_{1,N}^{(0,0)}} (\partial_x u_N(x))^2 \, dx \int_{\xi_{0,N}^{(0,0)}}^{\xi_{1,N}^{(0,0)}} (1-x)^\alpha (1+x)^\beta \, dx
         &\lesssim N^{-2\beta} \int_{\xi_{0,N}^{(0,0)}}^{\xi_{1,N}^{(0,0)}} (\partial_x u_N(x))^2 \, dx \\
         &\leq N^{-2\beta} \| \partial_x u_N \|_{(0,0)}^2
         \lesssim \| \partial_x u_N \|_{(\alpha,\beta)}^2.
   \end{align*}
   The same argument with $\alpha$ and $\beta$ interchanged proves the bound for the last interval, and \eqref{eq:pw-lin-interp-stability} follows.
\end{proof}

\subsection{Weighted \texorpdfstring{$L^2$}{L2} estimates and mass preconditioning}
\label{sec:mass}

In this section, we consider the $L^2$ norm with Jacobi weight $(0,1)$ or $(1,0)$.
By symmetry, we henceforth consider only the $(0,1)$ case.
We are interested in diagonal preconditioners for the weighted mass matrix in the Gauss--Lobatto Lagrange basis.
The degree-$N$ mass matrix $M$ is given by
\begin{equation}
   \label{eq:mass-defn}
   M_{ij} = \int_{-1}^1 L_i(x) L_j(x) (1+x) \, dx,
\end{equation}
where $L_i$ and $L_j$ are the Lagrange polynomials defined by the nodes $\xi_{k,N}^{(0,0)}$.
The analysis is complicated by the use of $(0,0)$-weighted nodal points for the nodal basis functions, while the inner product in \eqref{eq:mass-defn} is weighted by the $(0,1)$ Jacobi weight.
We construct a diagonal preconditioner $D$ for $M$ by
\begin{equation}
   \label{eq:mass-diag-defn}
   D_{ii} = \widetilde{\rho}_i := \begin{cases}
      \rho_{0,N}^{(0,1)}, &\qquad i = 0, \\
      (1 + \xi_{i,N}^{(0,0)}) \rho_{i,N}^{(0,0)} & \qquad i \geq 1.
   \end{cases}
\end{equation}
For $u, v \in \PP^N$, let $\langle \cdot \, , \cdot \rangle_{\tilde{N}}$ denote the discrete inner product, and $\| \cdot \|_{\tilde{N}}$ the induced norm, associated with the diagonal matrix $D$,
\begin{align}
   \label{eq:discrete-inner-prod}
   \langle u, v \rangle_{\tilde{N}} &= \sum_{i=0}^N u_i(\xi_{i,N}^{(0,0)}) v_i(\xi_{i,N}^{(0,0)}) \tilde{\rho}_i, \\
   \label{eq:discrete-norm}
   \| u \|_{\tilde{N}}^2 &= \langle u, u \rangle_{\tilde{N}}.
\end{align}

We now turn to a lemma which is a useful tool to compute the approximation error of the discrete quadrature.
We first recall that the leading coefficient of the Jacobi polynomial is given by
\begin{equation}
   \label{eq:jacobi-leading}
   k_{n,n}^{(\alpha,\beta)} = \frac{\Gamma(2n + \alpha + \beta + 1)}{2^n n! \Gamma(n + \alpha + \beta + 1)}.
\end{equation}
Using the expression
\[
   P_n^{(\alpha,\beta)}(x)
      = \frac{\Gamma(n+\alpha+1)}{n! \Gamma(n + \alpha + \beta + 1)} \sum_{k=0}^n \binom{n}{k} \frac{\Gamma(n + k + \alpha + \beta + 1)}{2^k \Gamma(k + \alpha + 1)} (x-1)^k,
\]
the next coefficient is found to be
\begin{equation}
   \label{eq:jacobi-second}
   k_{n-1,n}^{(\alpha,\beta)}
      = k_{n,n}^{(\alpha,\beta)} \frac{n(\alpha - \beta)}{2n + \alpha + \beta}
      = \frac{(\alpha - \beta)\Gamma(2n + \alpha + \beta + 1)}{2^n (2n + \alpha + \beta) (n-1)! \Gamma(n + \alpha + \beta + 1)}.
\end{equation}

\begin{lemma}
   \label{lem:quadrature-error}
   Let $u \in \PP^{2N+1}$, $u(x) = \sum_{i=0}^{2N+1} u_i x^i$.
   Then,
   \begin{equation}
      \label{eq:quadrature-error}
      \sum_{i=0}^N u(\xi_{i,N}^{(0,0)}) \rho_{i,N}^{(0,0)}
         = \int_{-1}^1 u(x)\,dx + u_{2N} \frac{2^{2N+1} N ((N-1)!)^2((N+1)!)^2}{(N+1)(2N+1)((2N)!)^2}.
   \end{equation}
\end{lemma}
\begin{proof}
   Define the polynomial $\Psi$ by
   \[
      \Psi(x) = u(x) + a (1-x^2) P_{N-1}^{(1,1)}(x)^2 + b (1-x^2) P_{N-1}^{(1,1)}(x) P_{N}^{(1,1)}.
   \]
   The coefficients $a$ and $b$ are chosen such that $\deg(\Psi) \leq 2N - 1$.
   To this end, set
   \begin{align*}
      a &= \frac{u_{2N}}{(k_{N-1,N-1}^{(1,1)})^2}
         = u_{2N} \frac{2^{2N-2}((N-1)!)^2((N+1)!)^2}{(2N)!}, \\
      b &= \frac{u_{2N + 1}}{(k_{N-1,N-1}^{(1,1)})k_{N,N}^{(1,1)}}
         = u_{2N + 1} \frac{2^{2N-1} (N-1)! N! (N+1)! (N+2)! }{ (2N+2)! (2N)! }.
   \end{align*}
   Note that the degree $2N$ coefficient of $(1 - x^2) P_{N-1}^{(1,1)}(x) P_N^{(1,1)}(x)$ is $k_{N-1,N-1}^{(1,1)} k_{N-1,N}^{(1,1)} + k_{N-2,N-1}^{(1,1)} k_{N,N}^{(1,1)} = 0$, since, by \eqref{eq:jacobi-second}, $k_{n-1,n}^{(1,1)} = 0$ for all $n$.
   Therefore, $a$ and $b$ cancel the degree $2N$ and $2N + 1$ terms in $u(x)$.
   Furthermore, $\Psi(\xi_{i,N}^{(0,0)}) = u(\xi_{i,N}^{(0,0)})$ for all $i$, since $\partial_x P_N^{(0,0)} = \frac{1}{2}(N+1)P_{N-1}^{(1,1)}$.
   By exactness of the Gauss--Lobatto quadrature,
   \begin{align*}
      \sum_{i=0}^N u(\xi_{i,N}^{(0,0)}) \rho_{i,N}^{(0,0)}
         &= \sum_{i=0}^N \Psi(\xi_{i,N}^{(0,0)}) \rho_{i,N}^{(0,0)}
         = \int_{-1}^1 \Psi(x) \, dx \\
         &= \int_{-1}^1 u(x) \, dx + a (P_{N-1}^{(1,1)}, P_{N-1}^{(1,1)})_{(1,1)} + b(P_{N-1}^{(1,1)}, P_{N}^{(1,1)})_{(1,1)} \\
         &= \int_{-1}^1 u(x) \, dx + a \frac{8N}{(2N+1)(N+1)},
   \end{align*}
   by \eqref{eq:jacobi-norm}; \eqref{eq:quadrature-error} follows.
\end{proof}

\begin{remark}
   The same aliasing proof technique of \Cref{lem:quadrature-error} can be used to compute the quadrature error of higher-degree integrands.
   However, for the purposes of mass matrix preconditioning, it is sufficient to consider polynomials of degree at most $2N + 1$.
\end{remark}

We introduce a basis for $\PP^{N}$ that will be useful in the following analysis.
Define $\{ \Phi_i \}_{i=0}^N$ by
\begin{equation}
   \label{eq:phi-basis-defn}
   \Phi_i(x) = \begin{cases}
      (1+x) P_i^{(0,3)}(x), &\qquad i \leq N-1, \\
      P_N^{(0,2)}(x), &\qquad i = N.
   \end{cases}
\end{equation}
This basis is orthogonal in the weighted inner product $(\cdot\,,\cdot)_{(0,1)}$.
Since $\Phi_i(-1) = 0$ for $i \leq N - 1$, by exactness of the Gauss--Lobatto quadrature \eqref{eq:lobatto-exact}, for $i + j \leq 2N - 4$, $i \neq j$, $\langle \Phi_i, \Phi_j \rangle_{\tilde{N}} = 0$.
By the same reasoning, $\Phi_N$ is discretely orthogonal to $\Phi_i$ for $i \leq N - 3$.
The discrete norm and inner product properties of this basis are summarized in the following result.

\begin{lemma}
   \label{lem:discrete-quad-properties}
   Defining $\Phi_i$ by \eqref{eq:phi-basis-defn}, the following identities and estimates hold.

   \noindent
   \begin{minipage}{0.5\textwidth}
      \begin{alignat}{3}
         \label{eq:phi-orth}
         &(\Phi_i, \Phi_j)_{(0,1)} = 0, \quad i \neq j, \\
         \label{eq:phi-i-norm}
         &\| \Phi_i \|_{(0,1)}^2 = \tfrac{8}{i + 2}, \quad i \leq N - 1, \\
         \label{eq:phi-n-norm}
         &\| \Phi_N \|_{(0,1)}^2 = 2, \\
         \label{eq:phi-i-disc-orth}
         &\langle \Phi_i, \Phi_j \rangle_{\tilde{N}} = 0, \quad i + j \leq 2N - 4, \quad i \neq j, \\
         \label{eq:phi-n-disc-orth}
         &\langle \Phi_i, \Phi_N \rangle_{\tilde{N}} = 0, \quad i \leq N - 3, \\
         \label{eq:phi-nm2-nm1}
         &\langle \Phi_{N-2}, \Phi_{N-1} \rangle_{\tilde{N}} = \tfrac{8 (N-1)}{(N+1)(N+2)},
      \end{alignat}
   \end{minipage}%
   \begin{minipage}{0.5\textwidth}
      \begin{alignat}{3}
         \label{eq:phi-nm2-n}
         &\langle \Phi_{N-2}, \Phi_{N} \rangle_{\tilde{N}} = \tfrac{8(N-1)}{N(N+2)}, \\
         \label{eq:phi-nm1-n}
         &\langle \Phi_{N-1}, \Phi_{N} \rangle_{\tilde{N}} = \tfrac{-8(N^2 - 5N - 5)}{(N+1)(N+2)^2}, \\
         \label{eq:phi-i-disc-norm}
         &\| \Phi_i \|_{\tilde{N}}^2 = \| \Phi_i \|_{(0,1)}^2, \qquad i \leq N - 2, \\
         \label{eq:phi-nm1-disc-norm}
         &\| \Phi_{N-1} \|_{\tilde{N}}^2 = \| \Phi_{N-1} \|_{(0,1)}^2 + \tfrac{72 N}{(N+1)(N+2)^2}, \\
         \label{eq:phi-n-disc-norm}
         &\| \Phi_N \|_{\tilde{N}}^2 = \| \Phi_N \|_{(0,1)}^2 + \tfrac{2(N^3 - 2N^2 + 16 N + 12)}{N(N+2)^2}.
      \end{alignat}
   \end{minipage}
\end{lemma}

These can be shown through repeated application of \Cref{lem:quadrature-error}; details of the proof are omitted for brevity.

\begin{lemma}
   \label{lem:disc-norm-equiv}
   The discrete norm $\| \cdot \|_{\tilde{N}}$ defined by \eqref{eq:discrete-norm} and the weighted $L^2$ norm $\| \cdot \|_{(0,1)}$ satisfy, for all $u \in \PP^N$,
   \begin{equation}
      \label{eq:disc-norm-equiv}
      \| u \|_{\tilde{N}}^2 \lesssim \| u \|_{(0,1)}^2 \lesssim N \| u \|_{\tilde{N}}^2.
   \end{equation}
\end{lemma}
\begin{proof}
   Expanding $u \in \PP^N$ in this basis, $u = \sum_{i=0}^N \widehat{u}_i \Phi_i$,
   \begin{equation}
      \label{eq:norm-u-series}
      \| u \|_{(0,1)}^2 = \sum_{i=0}^N \widehat{u}_i^2 \| \Phi_i \|_{(0,1)}^2.
   \end{equation}
   By the discrete orthogonality identities \eqref{eq:phi-i-disc-orth} and \eqref{eq:phi-n-disc-orth},
   \[
      \| u \|_{\tilde{N}}^2 = \sum_{i=0}^N \widehat{u}_i^2 \| \Phi_i \|_{\tilde{N}}^2 + 2 \widehat{u}_{N-2}\widehat{u}_{N-1} \langle \Phi_{N-2}, \Phi_{N-1} \rangle_{\tilde{N}} + 2 \widehat{u}_{N-2}\widehat{u}_{N} \langle \Phi_{N-2}, \Phi_{N} \rangle_{\tilde{N}} + 2 \widehat{u}_{N-1}\widehat{u}_{N} \langle \Phi_{N-1}, \Phi_{N} \rangle_{\tilde{N}}.
   \]
   Applying Young's inequality to the cross terms, and comparing \eqref{eq:phi-nm2-nm1}, \eqref{eq:phi-nm2-n}, and \eqref{eq:phi-nm1-n} with \eqref{eq:phi-i-disc-norm}, \eqref{eq:phi-nm1-disc-norm}, and \eqref{eq:phi-n-disc-norm}, we see
   \[
      \| u \|_{\tilde{N}}^2 \lesssim \sum_{i=0}^N \widehat{u}_i^2 \| \Phi_i \|_{\tilde{N}}^2.
   \]
   The lower bound in \eqref{eq:disc-norm-equiv} is then obtained by comparing the discrete norms with \eqref{eq:phi-i-norm} and \eqref{eq:phi-n-norm} and using \eqref{eq:norm-u-series}.

   To obtain the upper bound, we absorb the cross terms into $\| \Phi_{N-2} \|_{\tilde{N}}^2$, $\| \Phi_{N-1} \|_{\tilde{N}}^2$, and $\| \Phi_{N} \|_{\tilde{N}}^2$.
   From \eqref{eq:phi-nm2-nm1}, \eqref{eq:phi-i-disc-norm}, and \eqref{eq:phi-nm1-disc-norm},
   \begin{align*}
      \| \Phi_{N-2} \|_{\tilde{N}}^2 - \langle \Phi_{N-2}, \Phi_{N-1} \rangle_{\tilde{N}}
         &= \frac{3}{N + 2} \| \Phi_{N-2} \|_{\tilde{N}}^2, \\
      \| \Phi_{N-1} \|_{\tilde{N}}^2 - \langle \Phi_{N-2}, \Phi_{N-1} \rangle_{\tilde{N}}
         &= \frac{11 N^2 + 5N + 2}{N(N^2 + 13N + 4)} \| \Phi_{N-1} \|_{\tilde{N}}^2,
   \end{align*}
   so by Young's inequality
   \begin{equation*}
      \mathtoolsset{multlined-width=0.75\displaywidth}
      \begin{multlined}
         \| u \|_{\tilde{N}}^2 \gtrsim \sum_{i=0}^{N-3} \widehat{u}_i^2 \| \Phi_i \|_{\tilde{N}}^2
            + \frac{1}{N} (\widehat{u}_{N-2}^2 \| \Phi_{N-2} \|_{\tilde{N}}^2
            + \widehat{u}_{N-1}^2 \| \Phi_{N-1} \|_{\tilde{N}}^2)
            + \widehat{u}_N^2 \| \Phi_N \|_{\tilde{N}}^2 \\
            - | 2 \langle \widehat{u}_{N-2}\widehat{u}_{N} \langle \Phi_{N-2}, \Phi_{N} \rangle_{\tilde{N}} | -  | 2 \widehat{u}_{N-1}\widehat{u}_{N} \langle \Phi_{N-1}, \Phi_{N} \rangle_{\tilde{N}} |.
      \end{multlined}
   \end{equation*}
   By the weighted Young's inequality, \eqref{eq:phi-nm1-n}, and \eqref{eq:phi-nm2-n},
   \begin{align*}
      \langle \widehat{u}_{N-1}\widehat{u}_{N} \langle \Phi_{N-1}, \Phi_{N} \rangle_{\tilde{N}}
         &\lesssim \frac{\epsilon_1}{N} \widehat{u}_{N-1}^2 \| \Phi_{N-1} \|_{\tilde{N}}^2
         + \frac{1}{\epsilon_1 N} \widehat{u}_{N}^2 \| \Phi_{N} \|_{\tilde{N}}^2, \\
      \langle \widehat{u}_{N-2}\widehat{u}_{N} \langle \Phi_{N-2}, \Phi_{N} \rangle_{\tilde{N}}
         &\lesssim \frac{\epsilon_2}{N} \widehat{u}_{N-2}^2 \| \Phi_{N-2} \|_{\tilde{N}}^2
         + \frac{1}{\epsilon_2 N} \widehat{u}_{N}^2 \| \Phi_{N} \|_{\tilde{N}}^2.
   \end{align*}
   Since, by \eqref{eq:phi-n-disc-norm}, $\| \Phi_{N} \|_{\tilde{N}}^2 \approx 1$, the weights $\epsilon_1$ and $\epsilon_2$ can be chosen such that
   \begin{align*}
      \| u \|_{\tilde{N}}^2
         &\gtrsim \sum_{i=0}^{N-3} \widehat{u}_i^2 \| \Phi_i \|_{\tilde{N}}^2
            + \frac{1}{N} (\widehat{u}_{N-2}^2 \| \Phi_{N-2} \|_{\tilde{N}}^2
            + \widehat{u}_{N-1}^2 \| \Phi_{N-1} \|_{\tilde{N}}^2)
            + \widehat{u}_N^2 \| \Phi_N \|_{\tilde{N}}^2 \\
         &\geq \frac{1}{N} \sum_{i=0}^{N} \widehat{u}_i^2 \| \Phi_i \|_{\tilde{N}}^2
         \approx \frac{1}{N} \sum_{i=0}^{N} \widehat{u}_i^2 \| \Phi_i \|_{(0,1)}^2
         = \frac{1}{N} \| u \|_{(0,1)}^2. \qedhere
   \end{align*}
\end{proof}

The spectrum of the mass matrix $M$ can be estimated using the Rayleigh quotient $\frac{ \| u \|_{(0,1)}^2}{\sum_{i=0}^N u(\xi_{i,N}^{(0,0)})^2}$.
Taking $u(x) = 1$ gives $\lambda_{\max}(M) \gtrsim N^{-1}$, and taking $u(x) = \Phi_N(x)$ gives $\lambda_{\min}(M) \lesssim N^{-4}$;
the unpreconditioned mass matrix has condition number $\kappa(M) \approx N^3$.
However, as a consequence of \Cref{lem:disc-norm-equiv}, the diagonally preconditioned mass matrix has condition number $\kappa(D^{-1}M) \approx N$.

\begin{corollary}
   \label{cor:mass-gen-eig}
   The generalized eigenvalues of $M$ and $D$ satisfy
   \[
      \bm u^\tr D \bm u \lesssim \bm u^\tr M \bm u \lesssim N \bm u^\tr D \bm u.
   \]
\end{corollary}

From the preceding results, we do not expect the high-order Jacobi-weighted mass matrix to be spectrally equivalent to its low-order-refined counterpart, independent of $N$.
However, we can prove weighted $L^2$ equivalence of the high-order and low-order-refined interpolants with certain additional assumptions.

\begin{lemma}
   \label{lem:inv-weighted-l2}
   Let $u_N \in \PP^N$ with $u_N(1) = 0$.
   Let $u_h \in V_{h,N}^{(0,0)}$ be such that $u_h(\xi_{i,N}^{(0,0)}) = u_N(\xi_{i,N}^{(0,0)})$ for all $i$.
   Then,
   \[
      \| u_N \|_{(-1,0)} \approx \| u_h \|_{(-1,0)}.
   \]
\end{lemma}
\begin{proof}
   Since $u_N(1) = 0$, $(1-x)^{-1} u_N(x)^2$ is a polynomial of degree at most $2N - 1$, and
   \[
      \| u_N \|_{(-1,0)}^2
         = \int_{-1}^1 (1-x)^{-1} u_N(x)^2 \, dx
         = \int_{-1}^1 (1-x) q(x)^2 \, dx
         = \sum_{i=0}^{N-1} (1 - \xi_{i,N}^{(0,0)})^{-1} u_N(\xi_{i,N}^{(0,0)})^2 \rho_{i,N}^{(0,0)}
   \]
   by \eqref{eq:lobatto-exact}.
   Let $h_i$ denote the width of the interval $[\xi_{i,N}^{(0,0)}, \xi_{i+1,N}^{(0,0)}]$, and on the same interval, let $m_i$ denote the minimum of $(1-x)^{-1}$, and let $M_{i}$ denote its maximum.
   Note that, from \eqref{eq:sundermann} and \eqref{eq:szego-weight}, $h_i \approx \rho_{i,N}^{(0,0)} \approx \rho_{i+1,N}^{(0,0)}$.
   Furthermore, for $i < N$, $(1 - \xi_{i,N}^{(0,0)})^{-1} = m_i \approx M_i = (1 - \xi_{i+1,N}^{(0,0)})$.
   Then,
   \begin{align*}
      \| u_h \|_{(-1,0)}^2
         = \sum_{i=0}^{N-1} \int_{\xi_{i,N}^{(0,0)}}^{\xi_{i+1,N}^{(0,0)}} (1-x)^{-1} u_h(x)^2 \, dx
         \gtrsim \sum_{i=0}^{N-1} m_i h_i (u_h(\xi_{i,N}^{(0,0)})^2 + u_h(\xi_{i+1,N}^{(0,0)})^2)
         \approx \| u_N \|_{(-1,0)}^2.
   \end{align*}
   Similarly,
   \begin{align*}
      \| u_h \|_{(-1,0)}^2
         &\lesssim \sum_{i=0}^{N-2} M_i h_i (u_h(\xi_{i,N}^{(0,0)})^2 + u_h(\xi_{i+1,N}^{(0,0)})^2)
            + \int_{\xi_{N-1}^{(0,0)}}^{\xi_{N}^{(0,0)}} (1-x)^{-1} u_h(x)^2 \, dx \\
         &\lesssim \| u_N \|_{(-1,0)}^2 + \int_{\xi_{N-1}^{(0,0)}}^{\xi_{N}^{(0,0)}} (1-x)^{-1} u_h(x)^2 \, dx.
   \end{align*}
   Since $u_h(1) = 0$, on the endpoint interval $[\xi_{N-1}^{(0,0)}, \xi_{N}^{(0,0)}]$, $u_h(x) = u_N(\xi_{N-1,N}^{(0,0)}) (1-x) (1 - \xi_{N-1,N}^{(0,0)})^{-1}$, and
   \[
      \int_{\xi_{N-1}^{(0,0)}}^{\xi_{N}^{(0,0)}} (1-x)^{-1} u_h(x)^2 \, dx
         = \frac{1}{2} u_N(\xi_{N-1}^{(0,0)})^2 \lesssim \| u_N \|_{(-1,0)}^2,
   \]
   since $(1 - \xi_{i,N}^{(0,0)})^{-1} \rho_{i,N}^{(0,0)} \approx 1$.
\end{proof}

\begin{lemma}
   \label{lem:weighted-l2}
   Let $u_N \in \PP^{N-1}$, and let $u_h \in V_{h,N}^{(0,0)}$ be such that $u_h(\xi_{i,N}^{(0,0)}) = u_N(\xi_{i,N}^{(0,0)})$ for all $i$.
   Then, for $\alpha,\beta \in \{ 0, 1 \}$,
   \[
      \| u_N \|_{(\alpha,\beta)} \approx \| u_h \|_{(\alpha,\beta)}.
   \]
\end{lemma}
\begin{proof}
   First note that since $u_N \in \PP^{N-1}$, $\deg(u_N^2 (1-x)^\alpha (1+x)^\beta) \leq 2N$, and so, by \eqref{eq:lobatto-bound},
   \begin{equation}
      \label{eq:un-lobatto-bound}
      \| u_N \|_{(\alpha,\beta)}^2 \approx \sum_{i=0}^N u_N(\xi_{i,N}^{(0,0)})^2 (1-\xi_{i,N}^{(0,0)})^\alpha (1+\xi_{i,N}^{(0,0)})^\beta \rho_i^{(0,0)}.
   \end{equation}
   Let $h_i$ denote the width of the interval $[\xi_{i,N}^{(0,0)}, \xi_{i+1,N}^{(0,0)}]$, and on the same interval, let $m_i$ denote the minimum of $(1-x)^\alpha (1+x)^\beta$, and let $M_{i}$ denote its maximum.
   Note that for $i \neq 0$, $i \neq N-1$, $m_i \approx m_{i+1}$ and $M_i \approx M_{i+1}$.
   Then,
   \begin{align*}
      \| u_h \|_{(\alpha,\beta)}^2
         &= \sum_{i=0}^{N-1} \int_{\xi_{i,N}^{(0,0)}}^{\xi_{i+1,N}^{(0,0)}} u_h(x)^2 (1-x)^\alpha (1+x)^\beta \, dx
         \geq \sum_{i=1}^{N-2} m_i \int_{\xi_{i,N}^{(0,0)}}^{\xi_{i+1,N}^{(0,0)}} u_h(x)^2  \, dx \\
         &\gtrsim \sum_{i=1}^{N-2} m_i h_i (u_h(\xi_{i,N}^{(0,0)})^2 + u_h(\xi_{i+1,N}^{(0,0)})^2)
         \gtrsim \| u_N \|_{(\alpha,\beta)}^2,
   \end{align*}
   by \eqref{eq:un-lobatto-bound}, using that $h_i \approx \rho_{i,N}^{(0,0)}$.
   For the upper bound,
   \begin{align*}
      \| u_h \|_{(\alpha,\beta)}^2
         &= \sum_{i=0}^{N-1} \int_{\xi_{i,N}^{(0,0)}}^{\xi_{i+1,N}^{(0,0)}} u_h(x)^2 (1-x)^\alpha (1+x)^\beta \, dx
         \leq \sum_{i=0}^{N-1} M_i \int_{\xi_{i,N}^{(0,0)}}^{\xi_{i+1,N}^{(0,0)}} u_h(x)^2  \, dx \\
         &\lesssim \sum_{i=1}^{N-2} M_i h_i (u_h(\xi_{i,N}^{(0,0)})^2 + u_h(\xi_{i+1,N}^{(0,0)})^2)
            + M_0 h_0 u_h(\xi_{0,N}^{(0,0)})^2 + M_N h_N u_h(\xi_{N,N}^{(0,0)})^2 \\
         &\lesssim \| u_N \|_{(\alpha,\beta)}^2 + M_0 h_0 u_h(\xi_{0,N}^{(0,0)})^2 + M_N h_N u_h(\xi_{N,N}^{(0,0)})^2.
   \end{align*}
   It remains to show that the terms $M_0 h_0 u_h(\xi_{0,N}^{(0,0)})^2$ and $M_{N-1} h_{N-1} u_h(\xi_{N,N}^{(0,0)})^2$ are bounded by $\| u_N \|_{(\alpha,\beta)}^2$.
   We consider the first term; the second term will follow analogously.
   If $\beta = 0$, this is an immediate consequence of \eqref{eq:un-lobatto-bound}, so we assume $\beta = 1$.
   Then, $M_0 h_0 \approx N^{-4}$.
   Since $u_N \in \PP^{N-1}$, the integral in $\| u_N \|_{(\alpha,\beta)}^2$ can be computed exactly using $(\alpha,\beta)$-Jacobi--Gauss--Lobatto quadrature, i.e.
   \[
      \| u_N \|_{(\alpha,\beta)}^2
         = \sum_{i=0}^N u_N(\xi_{i,N}^{(\alpha,\beta)})^2 \rho_{i,N}^{(\alpha,\beta)}
         \geq u_N(\xi_{0,N}^{(\alpha,\beta)})^2 \rho_{0,N}^{(\alpha,\beta)}
         \gtrsim N^{-4} u_N(\xi_{0,N}^{(\alpha,\beta)})^2,
   \]
   since
   \[
      \rho_{0,N}^{(\alpha,\beta)} = \frac{2^{\alpha + \beta + 1} (\beta + 1) \Gamma^2(\beta + 1) \Gamma(N) \Gamma(N + \alpha + 1)}{\Gamma(N + \beta + 1) \Gamma(N + \alpha + \beta + 2)}
      = \frac{2^{\alpha + 3}}{N(N-1)(N+\alpha)(N + \alpha + 1)}. \qedhere
   \]
   From this, it holds that $\| u_h \|_{(\alpha,\beta)}^2 \lesssim \| u_N \|_{(\alpha,\beta)}^2$, and the conclusion follows.
\end{proof}

\subsection{Choice of nodal lattice points}
\label{sec:lattice-points}

The degrees of freedom specified in \Cref{thm:unisolvent} form a unisolvent set for any choice of distinct lattice points \eqref{eq:lattice-points}.
However, the choice of nodal point locations has an important impact on interpolation accuracy, conditioning of the resulting system matrices, and construction of preconditioners.
The preceding results concern the use of the Gauss--Lobatto abscissas $x_i = y_i = \xi_{i,N}^{(0,0)}$ as nodal interpolation points.
The polynomial interpolation stability results are the main ingredient required to prove $N$-independent spectral equivalence for the low-order-refined systems.
Additionally, the use of Gauss--Lobatto points as nodal points for quadrilateral elements is common.
If these points are also used for the lattice on triangular elements, then interelement continuity on mixed meshes with both triangular and quadrilateral elements may be enforced simply and naturally in the standard way by identifying global degrees of freedom.

In certain contexts, it may be useful to consider different choices of lattice points.
For nonconforming elements (as in discontinuous Galerkin discretizations), the identification of global dofs is not required, and so the requirement that the nodal points coincide on element interfaces may be relaxed.
This allows the use of Jacobi--Gauss--Radau points in $y$, and Jacobi--Gauss points in $x$.
Fully interior points such as Jacobi--Gauss quadrature are disadvantageous since the space compatibility conditions are most naturally given in terms of the right endpoint value.
The $(0,1)$-Jacobi--Gauss-Radau points give rise to a consistent mass matrix that is fully diagonal.

The use of Jacobi--Gauss--Lobatto points with Jacobi weights other than $(0,0)$ also presents advantages in the context of explicit time integration.
For example, the $(0,1)$-Jacobi--Gauss--Lobatto points give rise to a mass matrix, which, after diagonal preconditioning has condition number $\mathcal{O}(1)$.
Furthermore, the polynomial interpolation operator is stable in the $(0,1)$-weighted $H^1$ seminorm, cf.\ \Cref{thm:interp-stability}.
However, the downside of this choice of point set is that they are not distributed symmetrically about the origin, making them unsuitable for use with conforming elements on unstructured meshes.

The $(1,1)$-Jacobi--Gauss--Lobatto points are symmetric about the origin, and the associated mass matrix, after diagonal preconditioning, has $\mathcal{O}(1)$ condition number, as can be seen through an argument similar to the proof of \Cref{cor:mass-gen-eig} and \Cref{lem:disc-norm-equiv}.
However, the polynomial interpolation operator at these points is not uniformly stable in the $(0,1)$-weighted $H^1$ seminorm.

\section{Numerical results}
\label{sec:numerical}

In this section, we numerically study the finite elements and preconditioners described in the preceding sections.
We confirm the analytical results, and study the performance of the preconditioner applied to large, unstructured meshes, using algebraic multigrid as an approximate solver for the low-order-refined system.
For the large-scale test cases in \Cref{sec:meshes,sec:fictitious}, we have implemented the finite elements are collapsed lattice low-order-refined spaces in the framework of the MFEM finite element library \cite{Anderson2020,Andrej2024}.
The algebraic multigrid methods are provided by the \textit{hypre} library \cite{Falgout2002}.

\subsection{One-dimensional estimates}

In this section, we numerically study the constants in the norm equivalences of \Cref{lem:norm-equivalences} and \Cref{cor:norm-equivalences}.
For each of the equivalences, we numerically compute the best constants in the upper and lower bounds for given polynomial degree $N$ by forming the matrix that induces the corresponding norms or seminorms, and then computing the maximum and minimum generalized eigenvalues.
In the case of the weighted $H^1$ seminorms, we first remove the one-dimensional nullspace corresponding to the constant function.
The computed constants for $2 \leq N \leq 256$ are shown in \Cref{fig:1d-bounds}.
In all cases, we see that the lower bound constants are greater than 0.3 and the upper bound constants are less than 3.1.
The ratios of the upper bound to the lower bound ares bounded by 7.05.
The largest ratios occur for the estimate \eqref{eq:inv-weighted-l2-equiv}, $\| u_N \|_{(-1,0)}^2 \approx \| u_h \|_{(-1,0)}^2$ (where $u_N(1) = u_h(1) = 0$).

\begin{figure}
   \centering
   \hspace*{\fill}
   \includegraphics{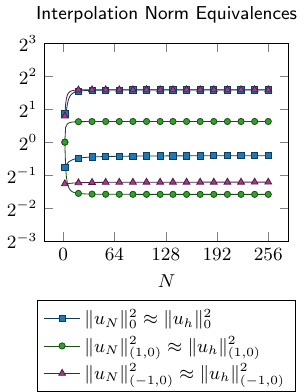}%
   \hspace*{\fill}%
   \includegraphics{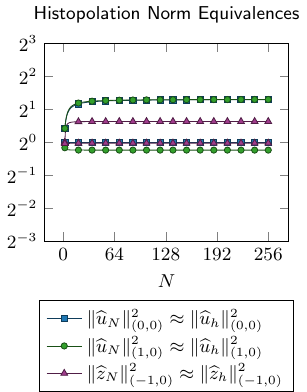}%
   \hspace*{\fill}
   \caption{Computed constants for the upper and lower bounds of the 1D norm equivalences in \Cref{lem:norm-equivalences} and \Cref{cor:norm-equivalences}.}
   \label{fig:1d-bounds}
\end{figure}

\subsection{Condition numbers on the triangle}
\label{sec:condition-number}

We presently consider the condition numbers $\kappa(A_0^{-1} A)$ on the reference triangle $\Delta$, for $A \in \{ A_{\Vh}, A_{\Wh}, A_{\Zh} \}$ and $A_0$ the corresponding low-order-refined matrix in $\{ A_{V_0}, A_{\bm W_0}, A_{Z_0} \}$.
For $2 \leq N \leq 128$, we compute these matrices, and compute the condition number by computing the extremal eigenvalues of $A_0^{-1} A$.
The eigenvalues are computed on the orthogonal complement of the nullspace (i.e.\ the constant functions for $A_{\Vh}$ and the irrotational functions for $A_{\Wh}$).
In the case of $A_{\Wh}$, this is performed using the singular value decomposition of the discrete curl matrix $C_{\bm W}$.
The results are shown in the left panel of \Cref{fig:tri-cond}.

Note that since, as described in the proof of \Cref{thm:spectral-equivalence}, the $\Hcurl$ stiffness matrices can be written as $A_{\Wh} = C_{\bm W}^T A_{Z_h} C_{\bm W}$ and $A_{\bm W_0} = C_{\bm W}^T A_{Z_0} C_{\bm W}$, it holds that $\kappa(A_{\bm W_0}^{-1}A_{\bm W_h}) \leq \kappa(A_{Z_0}^{-1}A_{Z_h})$ for all $N$.
Since the discrete curl is surjective, we have in fact $\kappa(A_{\bm W_0}^{-1}A_{\bm W_h}) = \kappa(A_{Z_0}^{-1}A_{Z_h})$; this is confirmed in the numerical results.

\begin{figure}
   \centering
   \hspace*{\fill}%
   \includegraphics{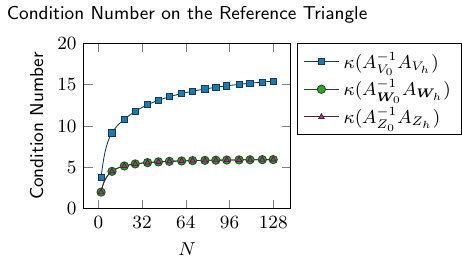}%
   \hspace*{\fill}%
   \includegraphics{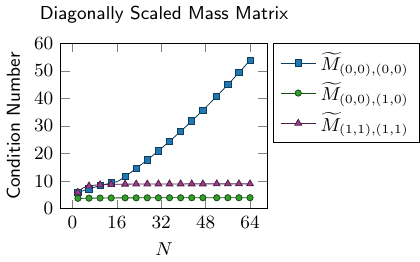}%
   \hspace*{\fill}

   \caption{Left: condition number of preconditioned matrices $A_0^{-1}A$ on the reference triangle $\Delta$. Right: diagonally preconditioned $H^1$ mass matrices.}.
   \label{fig:tri-cond}
\end{figure}

We also consider the condition number of the $H^1$ mass matrix on the space $V_h$, scaled by the inverse of its diagonal.
We let $\widetilde{M}_{(\alpha_x,\beta_x),(\alpha_y,\beta_y)} = D_{(\alpha_x,\beta_x),(\alpha_y,\beta_y)}^{-1} M_{(\alpha_x,\beta_x),(\alpha_y,\beta_y)}$, where $(\alpha_x,\beta_x)$ denotes the Jacobi weights of the interpolation points in the $x$ direction of the lattice, and $(\alpha_y, \beta_y)$ denotes the Jacobi weights in the $y$ direction;
see \Cref{sec:lattice-points} for a discussion of these choices.
The Duffy transformation induces a weight of $(1-y)$ in the mass matrix integral.
By \Cref{cor:mass-gen-eig}, using Gauss--Lobatto points in the $y$ direction will incur a factor of $N$ in the condition number, and by \Cref{lem:norm-equivalences}, the factor from the $x$ direction is $\mathcal{O}(1)$.
This is confirmed by the results shown in the right panel of \Cref{fig:tri-cond}.
The triangular mass matrix using $(0,0),(0,0)$ interpolation points exhibits linear growth in the condition number with respect to $N$.
By the Jacobi--Gauss--Lobatto quadrature bound \eqref{eq:lobatto-bound}, using $(1,0)$ points in the $y$ direction and retaining $(0,0)$ points in the $x$ direction should result in asymptotically constant condition number.
This too is confirmed in the numerical results.
However, this choice may not be practical, since the $(1,0)$ points are not symmetric, and using different points in the $x$ and $y$ directions may result in unaligned nodes on triangular meshes.
For this reason, we also consider using $(1,1)$ points in both directions.
As mentioned in \Cref{sec:lattice-points}, by an argument similar to that of \Cref{lem:disc-norm-equiv}, it can be shown that this set of points also results in uniformly bounded condition numbers, as is confirmed in the numerical results.

\subsection{Unstructured meshes and algebraic multigrid}
\label{sec:meshes}

\begin{figure}
   \centering

   \begin{tabular}{ccc}
      \includegraphics[height=3cm]{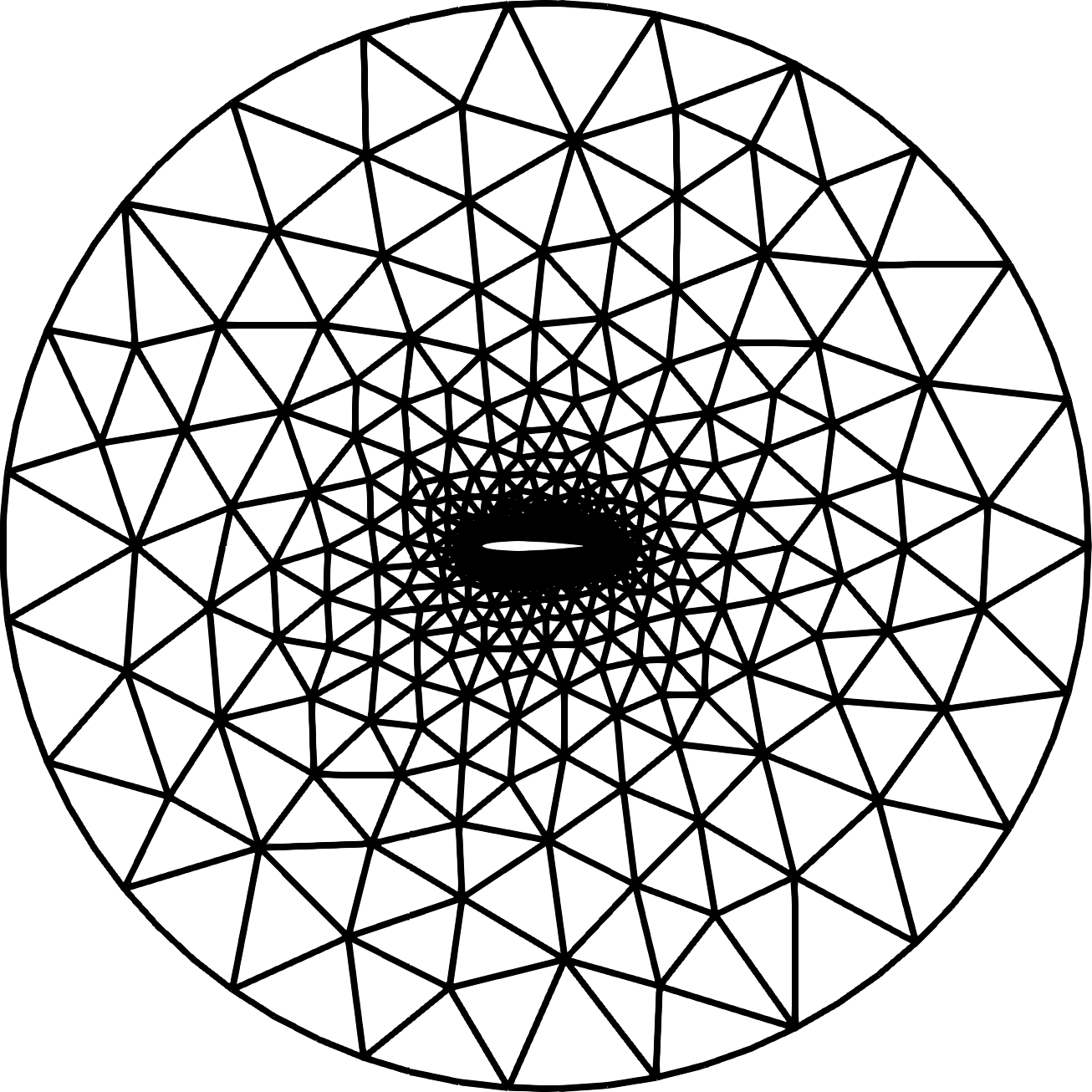} &
      \includegraphics[scale=0.9]{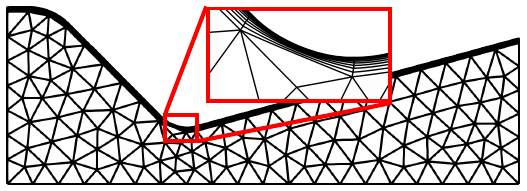} &
      \includegraphics[height=2.9cm]{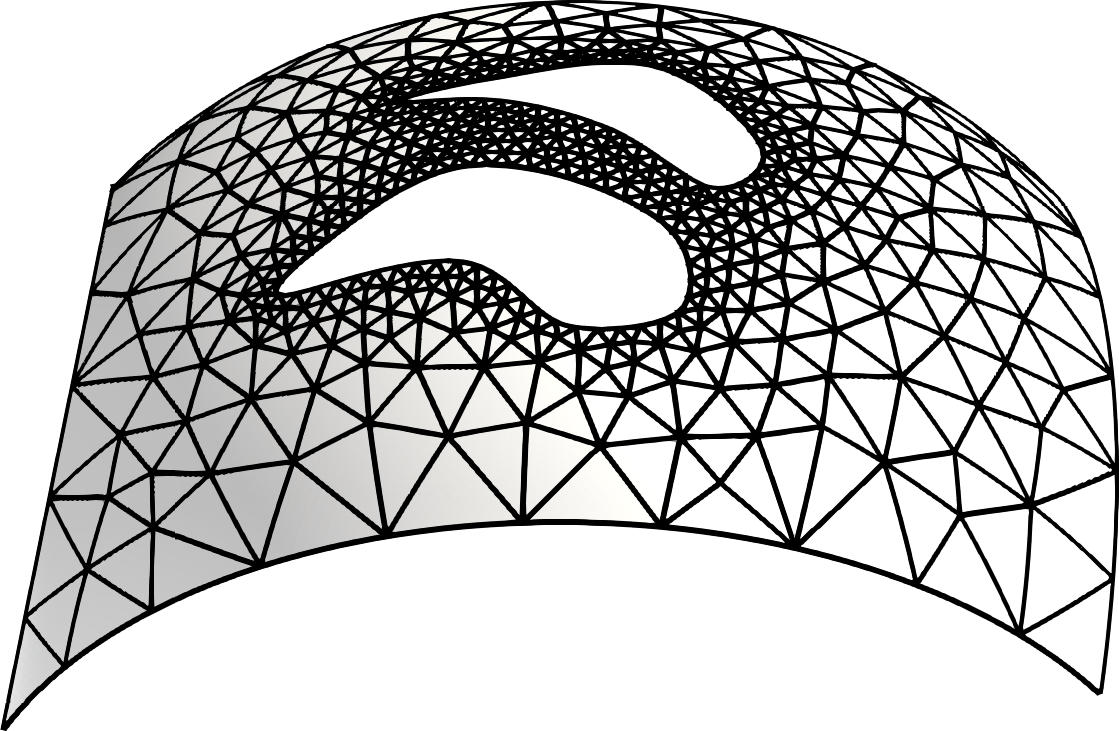} \\[10pt]
      \includegraphics[height=3cm]{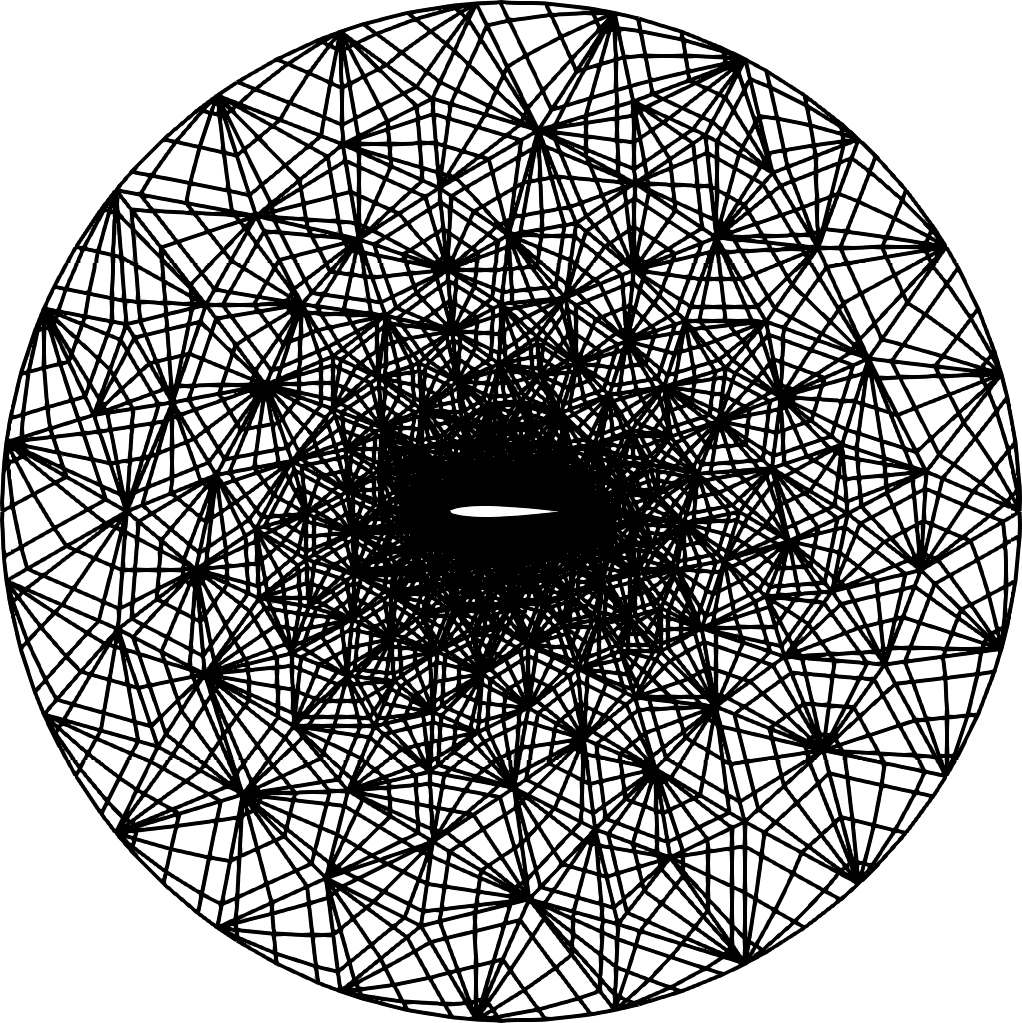} &
      \includegraphics[scale=0.9]{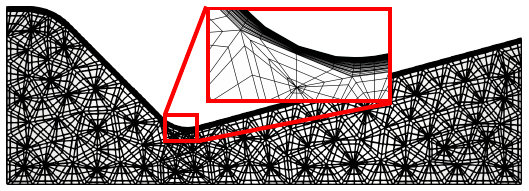} &
      \includegraphics[height=2.9cm]{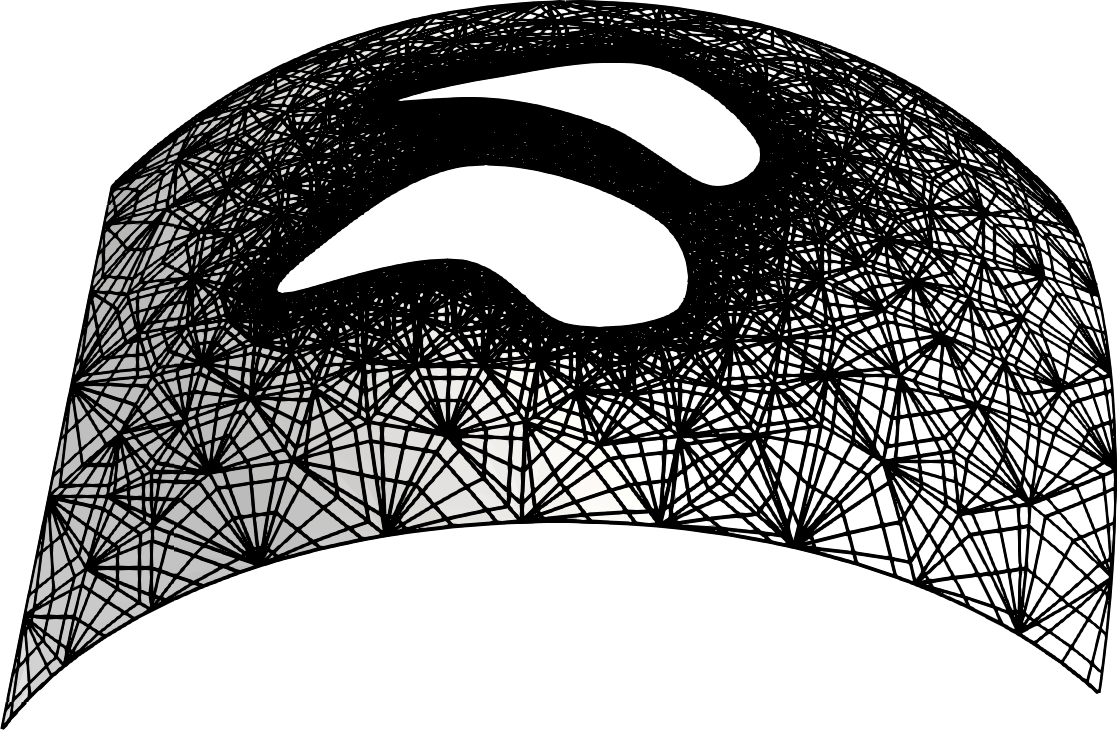} \\[10pt]
      (a) NACA0012 & (b) BGM45-15 nozzle & (c) Spherical surface patch
   \end{tabular}

   \caption{
      Meshes used for \Cref{sec:meshes} test cases.
      Top row: original meshes.
      Bottom row: low-order-refined meshes.
      An inset of the curved boundary layer elements is shown in panel (b).
   }
   \label{fig:meshes}
\end{figure}

In this section, we numerically study the performance of the low-order-refined preconditioner on several representative unstructured meshes, intended to assess the performance of the method in a variety of configurations.
The meshes considered, together with their low-order-refined variants, are shown in \Cref{fig:meshes}.
They include:
(a) a triangular mesh of the NACA0012 airfoil, with degree-3 curved elements at the surface of the airfoil and at the outer circular boundary;
(b) a mixed mesh of the BGM45-15 nozzle \cite{Back1965}, with anisotropic quadrilateral elements at the nozzle surface, and isotropic triangular elements away from the boundary layer; and, (c) a graded surface mesh consisting of a spherical patch with holes.

In each of the cases considered, we solve the problem
\begin{equation}
   \label{eq:mass-plus-laplace}
   u - \Delta u = f \qquad \text{in $\Omega$}
\end{equation}
with homogeneous Dirichlet boundary conditions on $\partial\Omega$ by looking for $u_h \in \Vh$ such that
\[
   a(u_h, v_h) := (u_h, v_h) + (\nabla u_h, \nabla v_h) = (f, v_h)
\]
for all $v_h \in V_h$.
Let $A$ denote the system matrix associated with the bilinear form $a(\cdot \, , \cdot)$ on $\Vh$, and let $A_0$ denote the system matrix associated with the same form on $V_0$.
We solve the system using the conjugate gradient method with a relative tolerance of $10^{-10}$.
As a preconditioner for $A$, we consider either $A_0^{-1}$ or $B_0$, where $A_0^{-1}$ is computed using a sparse direct factorization, and $B_0$ represents one V-cycle of algebraic multigrid applied to $A_0$.
We use the MUMPS library to compute the factorization of $A_0$ \cite{Amestoy2001}.
The V-cycle $B_0$ is computed using \textit{hypre}'s BoomerAMG \cite{Henson2002}.
We use Falgout coarsening with modified classical interpolation, two sweeps of symmetric $\ell^1$ hybrid Gauss--Seidel smoothers, no aggressive coarsening levels, and a strength threshold of 0.5.
In the results below, we report the number of iterations using both $A_0^{-1}$ and $B_0$.
The low-order-refined mesh results in stretched, anisotropic elements and clustering at the collapsed vertices;
consequently, the system $A_0$ itself, despite its sparsity, is somewhat challenging for algebraic multigrid methods.
The application of ordered ILU and line smoothers to address this difficulty in the case of quadrilateral elements was discussed in \cite{Pazner2020a}.

For the first test case, we consider a triangular mesh of the NACA0012 airfoil, consisting of 1,846 triangular elements.
Curved degree-3 elements at the surface of the boundary and at the circular farfield boundary are used.
Although the analysis was performed for the case of affine element transformations, the low-order-refined preconditioner remains effective in cases with curved elements.
The convergence results for an order-refinement study are presented in \Cref{tab:naca}.
With increasing order, up to $N=32$, we see only a mild increase in the number of iterations, consistent with the condition numbers computed in \Cref{sec:condition-number}.
Consistent with the discussion in \Cref{rem:sparsity}, the average number of nonzeros per row of $A_0$ approaches 9.
This is in contrast to the high-order matrix $A$, for which the number of nonzeros per row scales as $N^2$.

\begin{table}
   \centering
   \caption{
      Convergence results for the NACA0012 mesh.
      Iteration counts are reported for $A_0^{-1}$ as preconditioner (using a direct solver) and, in parentheses, for an algebraic multigrid V-cycle $B_0$.
   }
   \label{tab:naca}

   \small

   \begin{tabular}{c c C{7} C{4.2} C{1.2}}
      \multicolumn{5}{c}{NACA0012, $N$ refinement} \\
      \toprule
      {$N$} & {\# It.} & {\# DOFs} & {$\nnz(A) / \text{row}$} & {$\nnz(A_{0}) / \text{row}$} \\
      \midrule
      2 & 18 (18) & 5766 & 11.68 & 8.12 \\
      4 & 23 (27) & 26300 & 28.46 & 8.62 \\
      8 & 25 (33) & 111672 & 86.17 & 8.82 \\
      16 & 30 (40) & 459632 & 297.66 & 8.91 \\
      32 & 38 (52) & 1864416 & 1104.66 & 8.96 \\
      \bottomrule
   \end{tabular}
   \hfill
   \begin{tabular}{c c C{8} C{4.2} C{1.2}}
      \multicolumn{5}{c}{NACA0012, $h$ refinement ($N = 8$)} \\
      \toprule
      {Ref.} & {\# It.} & {\# DOFs} & {$\nnz(A) / \text{row}$} & {$\nnz(A_{0}) / \text{row}$} \\
      \midrule
      0 & 25 (33) & 111672 & 86.17 & 8.82\\
      1 & 27 (36) & 444864 & 86.49 & 8.84\\
      2 & 28 (40) & 1775808 & 86.64 & 8.85\\
      3 & 28 (44) & 7095936 & 86.72 & 8.86\\
      4 & 27 (47) & 28369152 & 86.76 & 8.86\\
      \bottomrule
   \end{tabular}
\end{table}

For the NACA0012 case, we also consider an $h$-refinement study, where the polynomial degree is fixed at $N = 8$, and the mesh is uniformly refined up to four times.
Under $h$ refinement, the condition number of $A_0^{-1}A$ is expected to remain roughly constant, which is confirmed by the results shown in the right panel of \Cref{tab:naca}.
A mild increase in the number of iterations using the AMG preconditioner is observed.
This is on account of marginally slower AMG convergence when applied to $A_0$.
The number of nonzeros per row for the high-order and low-order system remain roughly constant, but the low-order system has about $10\times$ fewer nonzeros.

In \Cref{tab:timings} we compare the runtimes of a matrix-based and matrix-free solution algorithms, illustrating the benefit of low-order-refined preconditioning in the context of matrix-free methods.
We measure the matrix assembly time, the time to compute one matrix-vector product (or, in the case of matrix-free methods, one operator evaluation), and the total time to solution, including conjugate gradient iterations with low-order-refined algebraic multigrid preconditioning.
The computational complexity of standard (non sum-factorized) matrix assembly algorithms scales as $N^{3d}$ in $d$ spatial dimensions.
We observe that the asymptotic growth in assembly time (and total time to solution, which at high orders is dominated by assembly time) approaches 6.
The number of nonzeros in the system matrix and the number of operations to compute a matrix-vector product scale as $N^2$; this too is confirmed in our results.
In contrast, the matrix-free assembly algorithm requires $\mathcal{O}(N^d)$ operations, and the matrix-vector product requires $\mathcal{O}(N^{d+1})$ operations.
These scalings are confirmed in our results.
Asymptotically, the time-to-solution using matrix-free methods scales as $N^3$.

\begin{table}
   \caption{
      Comparison of runtimes (in seconds) of matrix-based and matrix-free methods for the NACA0012 test case.
   }
   \label{tab:timings}
   \centering
   \small
   \begin{tabular}{c|lc|lc|lc}
      \multicolumn{7}{c}{Matrix-based runtimes} \\
      \toprule
      $N$ & \multicolumn{1}{c}{Assembly} & Rate & \multicolumn{1}{c}{Matvec} & Rate & \multicolumn{1}{c}{Total} & Rate \\
      \midrule
      2 & $ 2.84 \times 10^{-3} $ & --- & $ 7.80 \times 10^{-6} $ & --- & $ 2.12 \times 10^{-2} $ & --- \\
      4 & $ 1.14 \times 10^{-2} $ & 2.01 & $ 3.64 \times 10^{-5} $ & 2.22 & $ 6.34 \times 10^{-2} $ & 1.58 \\
      8 & $ 1.65 \times 10^{-1} $ & 3.88 & $ 8.56 \times 10^{-4} $ & 4.57 & $ 3.83 \times 10^{-1} $ & 2.65 \\
      16 & $ 5.16 \times 10^{0} $ & 4.99 & $ 1.77 \times 10^{-2} $ & 4.62 & $ 6.87 \times 10^{0} $ & 4.22 \\
      32 & $ 5.04 \times 10^{2} $ & 6.61 & $ 2.60 \times 10^{-1} $ & 3.88 & $ 5.23 \times 10^{2} $ & 6.25 \\
      \bottomrule
   \end{tabular}

   \vspace{12pt}

   \begin{tabular}{c|ccr|ccr|lcr}
      \multicolumn{10}{c}{Matrix-free runtimes} \\
      \toprule
      $N$ & Assembly & Rate & \multicolumn{1}{c|}{Speedup} & Matvec & Rate & \multicolumn{1}{c|}{Speedup} & \multicolumn{1}{c}{Total} & Rate & \multicolumn{1}{c}{Speedup} \\
      \midrule
      2 & $ 6.48 \times 10^{-4} $ & --- & $ 4.38 \times $ & $ 4.79 \times 10^{-5} $ & --- & $ 0.16 \times $ & $ 1.89 \times 10^{-2} $ & --- & $ 1.12 \times $ \\
      4 & $ 1.76 \times 10^{-3} $ & 1.44 & $ 6.50 \times $ & $ 7.35 \times 10^{-5} $ & 0.62 & $ 0.49 \times $ & $ 5.03 \times 10^{-2} $ & 1.41 & $ 1.26 \times $ \\
      8 & $ 4.88 \times 10^{-3} $ & 1.47 & $ 33.79 \times $ & $ 3.63 \times 10^{-4} $ & 2.39 & $ 2.36 \times $ & $ 1.75 \times 10^{-1} $ & 1.81 & $ 2.19 \times $ \\
      16 & $ 1.55 \times 10^{-2} $ & 1.68 & $ 331.88 \times $ & $ 2.54 \times 10^{-3} $ & 2.66 & $ 6.98 \times $ & $ 8.86 \times 10^{-1} $ & 2.06 & $ 7.76 \times $ \\
      32 & $ 5.77 \times 10^{-2} $ & 1.89 & $ 8732.11 \times $ & $ 1.74 \times 10^{-2} $ & 2.77 & $ 15.01 \times $ & $ 4.87 \times 10^{0} $ & 2.46 & $ 107.41 \times $ \\
      \bottomrule
   \end{tabular}
\end{table}

As a next test case, we consider the BGM45-15 nozzle geometry, which was studied in \cite{Back1965}.
As is common in applications with boundary layers, this mesh has several layers of thin, anisotropic quadrilateral elements adjacent to the top boundary.
Like in the previous case, curved degree-3 elements are used to accurately represent the boundary curvature.
Away from the boundary layer, isotropic triangles are used.
Quadrilaterals are often preferred for boundary layer elements since high-quality orthogonal elements can be constructed.
Away from the boundary, unstructured meshing algorithms often result in triangular elements, resulting in a mixed mesh.
The low-order-refined preconditioner can be applied to such methods in a straightforward manner.
The quadrilateral elements are refined using the square Gauss--Lobatto lattice, and the triangular elements are refined using the collapsed triangular lattice.
All edges are refined using Gauss--Lobatto points, so interface compatibility is automatic.
The convergence results for this case are presented in \Cref{tab:nozzle-surf}.
These results are similar to those of the NACA0012 test case.

For a final test case, we consider a topologically 2D surface mesh embedded in 3D space.
The geometry for this case is a patch of a spherical surface with two holes.
The mesh elements are graded to capture the smaller geometric features.
In this case, the Laplace--Beltrami operator is used in \eqref{eq:mass-plus-laplace}.
The results for this case are shown in \Cref{tab:nozzle-surf}.
The convergence results for all three test cases are quite similar.

\begin{table}
   \caption{
      Convergence results for the BGM45-15 nozzle and spherical surface patch mesh.
      Iteration counts are reported for $A_0^{-1}$ as preconditioner (using a direct solver) and, in parentheses, for an algebraic multigrid V-cycle $B_0$.
   }
   \label{tab:nozzle-surf}
   \centering
   \small

   \begin{tabular}{c c C{6} C{4.2} C{1.2}}
      \multicolumn{5}{c}{BGM45-15 nozzle, $N$ refinement} \\
      \toprule
      {$N$} & {\# It.} & {\# DOFs} & {$\nnz(A) / \text{row}$} & {$\nnz(A_{0}) / \text{row}$} \\
      \midrule
      2 & 25 (25) & 1481 & 13.59 & 8.40 \\
      4 & 27 (30) & 6225 & 31.70 & 8.72 \\
      8 & 30 (33) & 25505 & 92.02 & 8.86 \\
      16 & 34 (39) & 103233 & 308.71 & 8.93 \\
      32 & 39 (52) & 415361 & 1126.12 & 8.97 \\
      \bottomrule
   \end{tabular}
   \hfill
   \begin{tabular}{c c C{6} C{4.2} C{1.2}}
      \multicolumn{5}{c}{Spherical surface patch, $N$ refinement} \\
      \toprule
      {$N$} & {\# It.} & {\# DOFs} & {$\nnz(A) / \text{row}$} & {$\nnz(A_{0}) / \text{row}$} \\
      \midrule
      2 & 19 (19) & 2775 & 11.63 & 8.09 \\
      4 & 24 (27) & 12607 & 28.39 & 8.60 \\
      8 & 27 (34) & 53439 & 86.05 & 8.81 \\
      16 & 29 (40) & 219775 & 297.45 & 8.91 \\
      32 & 33 (52) & 891135 & 1104.26 & 8.95 \\
      \bottomrule
   \end{tabular}
\end{table}

\subsection{Fictitious space preconditioning}
\label{sec:fictitious}

The low-order-refined preconditioners considered in this work can be used to construct fictitious space preconditioners for the standard $\Pp_N$ finite element space that are robust with respect to the polynomial degree.
The fictitious space lemma of Nepomnyaschikh is summarized as follows \cite{Nepomnyaschikh2007,Nepomnyaschikh1991}.
\begin{lemma}[{\cite[Lemma 2.3]{Nepomnyaschikh1991}}]
   \label{lem:fictitious}
   Let $A : V \to V$, and $\widehat{A} : \widehat{V} \to \widehat{V}$ be self-adjoint linear operators on Hilbert spaces, and let $R : \widehat{V} \to V$ be a surjective linear operator.
   Suppose that, for all $\widehat{v} \in \widehat{V}$,
   \begin{equation}
      \label{eq:cR}
      \| R \widehat{v} \|_A^2 \leq c_R \| \widehat{v} \|_{\widehat{A}}^2,
   \end{equation}
   and, given $v \in \Vh$, there exists some $\widehat{v} \in \widehat{V}$ such that $v = R \widehat{v}$, and
   \begin{equation}
      \label{eq:cS}
      \| \widehat{v} \|_{\widehat{A}}^2 \leq c_S \| v \|_A^2,
   \end{equation}
   for some constants $c_R$ and $c_S$.
   Then, for $B = R \widehat{A}^{-1} R^\tr$,
   \[
      c_S^{-1} (A^{-1} v, v) \leq (Bv, v) \leq c_R (A^{-1} v, v),
   \]
   and so the condition number of $BA$ is bounded by
   \[
      \kappa(BA) \leq c_R c_S.
   \]
\end{lemma}
In this context, let $V$ denote the standard $\Pp_N$ Lagrange space, and let $\widehat{V} = V_h$ be the Duffy-mapped space defined by \eqref{eq:Vh}.
For $\widehat{v} \in \widehat{V}$, we define $v = R\widehat{v}$ to be the local elliptic projection of $\widehat{v}$, i.e.\ the unique element of $V$ satisfying, for each triangle $\kk \in \T$, $v = \widehat{v}$ on $\partial\kk$, and,
\[
   a(v,w)_{\kk} = a(\widehat{v}, w)_{\kk},
\]
for all test functions $w \in \Pp_{N,0}(\kk) = \{ w \in \Pp_N : w|_{\partial\kk} = 0 \}$ .
This mapping is well-defined since the traces of $\widehat{v}$ are polynomials of degree at most $N$ by \Cref{prop:h1-trace}.
Since $V \subseteq \widehat{V}$, the constant $c_S$ in \eqref{eq:cS} is equal to 1.
The constant $c_R$ in \eqref{eq:cR} is bounded, independent of $N$, by the existence of polynomial-preserving inverse trace operators \cite{Ainsworth2009}.
The fictitious system solver $\widehat{A}^{-1}$ can be replaced by a suitable approximation, e.g.\ one V-cycle of algebraic multigrid applied to the low-order-refined system $A_0$, as considered in the previous section.

In this context, one application of the preconditioner requires one application of algebraic multigrid applied to the sparse low-order-refined system, and two applications of the transfer operator $R$, which has block-diagonal structure, with blocks of size $(N-1)(N-2)/2$ for each element interior.
The block size grows quadratically with the polynomial degree, and so at very high orders the application of $R$ may be the dominant cost in the preconditioner application.

\begin{table}
   \caption{
      Convergence results for the fictitious space preconditioner.
      $B$ denotes the preconditioner using the high-order space $\Vh$ as a fictitious space, and $B_0$ denotes using the low-order refined preconditioner $V_0$.
      Iteration counts using algebraic multigrid for the fictitious system are shown in parentheses.
   }
   \label{tab:aux}
   \centering

   \begin{tabular}{c c c C{6} C{3}}
      \multicolumn{5}{c}{NACA0012} \\
      \toprule
      {$N$} & {\# It. $B$} &  {\# It. $B_0$} & {\# DOFs} & {Block size} \\
      \midrule
      2 & 13 (18) & 15 (16) & 3920  & 0 \\
      4 & 13 (20) & 19 (23) & 15224  & 3 \\
      8 & 13 (22) & 20 (28) & 59984  & 21 \\
      16 & 13 (27) & 20 (32) & 238112  & 105 \\
      32 & 14 (27) & 19 (35) & 948800  & 465 \\
      \bottomrule
   \end{tabular}
   \hspace{1cm}
   \begin{tabular}{c c c C{6} C{3}}
      \multicolumn{5}{c}{BGM45-15 nozzle} \\
      \toprule
      {$N$} & {\# It. $B$} &  {\# It. $B_0$} & {\# DOFs} & {Block size} \\
      \midrule
      2 & 18 (20) & 31 (32) & 1257 & 0 \\
      4 & 16 (19) & 25 (28) & 4881 & 3 \\
      8 & 15 (19) & 28 (32) & 19233 & 21 \\
      16 & 15 (26) & 29 (34) & 76353 & 105 \\
      32 & 16 (28) & 30 (37) & 304257 & 465 \\
      \bottomrule
   \end{tabular}
\end{table}

We repeat the degree-refinement studies of the previous section on the BGM45-15 and NACA0012 meshes, solving the standard Lagrange $\Pp_N$ system using the fictitious space preconditioner.
The results are shown in \Cref{tab:aux}.
These results confirm that the fictitious space preconditioner is robust with respect to polynomial degree.
The practical preconditioner, using algebraic multigrid applied to the low-order-refined system $A_0$, results in fewer than 40 conjugate gradient iterations in all cases considered.

\section{Conclusions}
\label{sec:conclusions}

In this work, we constructed high-order finite element spaces for the $L^2$ de Rham complex by pulling back appropriately chosen polynomial spaces on the unit square under the Duffy transformation.
Unisolvent degrees of freedom for these spaces are constructed using a collapsed lattice.
In the basis dual to these degrees of freedom, the discrete differential operators have a particularly simple matrix representation: they are equal to the oriented incidence matrix of the lattice.
The triangular lattice can also be used to construct a refined mesh, on which the lowest-order finite element spaces share the same degrees of freedom as the high-order space on the original mesh.
By establishing one-dimensional norm equivalences and stability estimates for the Gauss--Lobatto interpolation operator in Jacobi-weighted norms, we prove $N$-independent spectral equivalence of the high-order and low-order refined stiffness matrices.
Numerical results confirm these estimates on a variety of unstructured meshes.
Degree-robust preconditioners for the standard $\Pp_N$ finite element spaces can be constructed using the low-order-refined spaces through a fictitious approach.
The fictitious space preconditioner requires solving decoupled elliptic problems on the interior of each triangle; the development of efficient matrix-free methods for these local problems is an open problem of interest.
The extension of these approaches to three-dimensional meshes is a topic of future investigation.

\section*{Acknowledgements}

This work was partially supported by the U.S. Department of Energy, Office of Science, Office of Advanced Scientific Computing Research (ASCR) and NSF DMS-2136228.
Computational resources used in this work were provided by the Oregon Regional Computing Accelerator (Orca), funded by NSF CC*-2346732.
The author thanks V.~Dobrev, J.~Dong, and Tz.~Kolev for insightful discussions and comments.

\printbibliography

\end{document}